\documentclass[11pt,a4paper]{article}
\usepackage[latin1]{inputenc}
\usepackage{amsmath,amsthm,comment,psfrag}
\usepackage{amsfonts}
\usepackage{amssymb}
\usepackage{graphicx,xypic,pstricks}
\input xy
\xyoption{all}

\newcommand{\executeiffilenewer}[3]{%
 \ifnum\pdfstrcmp{\pdffilemoddate{#1}}%
 {\pdffilemoddate{#2}}>0%
 {\immediate\write18{#3}}\fi%
}

\theoremstyle{definition} 

 \newtheorem{definition}{Definition}[section]
 \newtheorem{remark}[definition]{Remark}
 \newtheorem{example}[definition]{Example}

\theoremstyle{plain}      

 \newtheorem{proposition}[definition]{Proposition}
 \newtheorem{theorem}[definition]{Theorem}
 \newtheorem{corollary}[definition]{Corollary}
 \newtheorem{lemma}[definition]{Lemma}
 \newtheorem{hypothesis}{Hypothesis}

\newtheorem*{theorem*}{Theorem}

\renewcommand{\epsilon}{\varepsilon}
\renewcommand{\phi}{\varphi}
\newcommand{\N}{\mathbb{N}}
\newcommand{\C}{\mathbb{C}}

\newcommand{\Z}{\mathbb{Z}}
\newcommand{\R}{\mathbb{R}}
\renewcommand{\P}{\mathbb{P}}

\newcommand{\hol}{\psi}
\newcommand{\la}{\langle}
\newcommand{\ra}{\rangle}
\newcommand{\dd}{\mathrm{d}}
\newcommand{\boP}{\mathcal{P}}
\newcommand{\boE}{\mathcal{E}}
\newcommand{\bog}{\mathcal{G}}
\newcommand{\boH}{\mathcal{H}}
\newcommand{\boU}{\mathcal{U}}
\newcommand{\boM}{\mathcal{M}}

\DeclareMathOperator{\tr}{Tr}

\DeclareMathOperator{\id}{Id}
\DeclareMathOperator{\pf}{Pfaff}
\newcommand{\shom}{\underline{\textrm{Hom}}}
\renewcommand{\hom}{\textrm{Hom}}

\newcommand{\ba}[1]{\overline{#1}}
\newcommand{\smalltheta}[2]{\raisebox{-#2cm}{\includegraphics[width=#1cm]{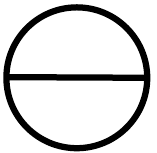}}}
\newcommand{\smalltet}[2]{\raisebox{-#2cm}{\includegraphics[width=#1cm]{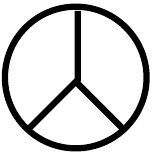}}}

\title{Generating series and asymptotics of classical spin networks}
\date{}
\author{F. Costantino\footnote{Institut de Recherche Math\'ematique Avanc\'ee, 7 Rue Ren\'e Descartes, 67084 Strasbourg, France.} \,\,and J. March{\'e}\footnote{Centre de Math{\'e}matiques Laurent Schwartz, Ecole Polytechnique, route de Saclay, 91128 Palaiseau, France.}}
\begin{document}
\maketitle

{\abstract We study classical spin networks with group SU$_2$. In the first part, using gaussian integrals,  we compute their generating series in the case where the edges are equipped with holonomies; this generalizes Westbury's formula. In the second part, we use an integral formula for the square of the spin network and perform stationary phase approximation under some non-degeneracy hypothesis. This gives a precise asymptotic behavior when the labels are rescaled by a constant going to infinity.
}

\section{Introduction}

A classical spin network is a pair $(\Gamma,c) $ where $\Gamma$ is a trivalent graph equipped with a cyclic ordering of the edges around vertices and $c$ is a map from the edges to the natural numbers called \emph{coloring} and satisfying some simple conditions. According to Penrose \cite{penrose}, one may associate to such a pair a rational number $\la \Gamma,c\ra$ obtained by contracting a tensor with values in some representations of SU$_2$. 

When $\Gamma$ is a theta graph ($\smalltheta{ .5}{0.15}$) or a tetrahedron ($\smalltet{.5}{0.15}$), these quantities were introduced by Racah and Wigner in 1940-1950 for the study of atomic spectra \cite{wigner}. Later, Ponzano and Regge used it as a discrete model for gravity.
 Immediately, physicists were interested in the study of the asymptotic behavior of the quantity $\la \Gamma,kc\ra$ when $k$ goes to infinity. This behavior corresponds to a classical limit of quantum mechanics and is expected to be related to euclidean geometric quantities. 
 
 In the nineties physicists used spin networks in the spin foam models for quantum gravity \cite{baez} and the study of the asymptotical behavior was extended from the 3j (theta) and 6j (tetrahedron) to more complicated networks as the 9j (complete bipartite 3,3 graph) 10j, 15j (skeleta of the 4 and 5 simplices), see \cite{barrett,freidel}. 
Mathematicians got interested in spin networks through the work of Kirillov, Reshetikhin \cite{kr} and Kauffman \cite{k} who introduced and studied their ``quantum versions" which are a key ingredient in the construction of quantum invariants of knots and 3-manifolds. In this paper, we will not deal with these ``quantum versions" but restrict to the classical case. 

The first rigorous proof for the asymptotical behavior of the 6j-symbols were obtained (in the so-called euclidean case) by Roberts \cite{roberts} and then re-obtained and extended using different techniques \cite{charles,GaVV,wt}. For general graphs, Garoufalidis and Van der Veen proved that the generating series of the sequence $k\mapsto \langle \Gamma,kc\rangle$ is a G-function, implying that the sequence $\la \Gamma,kc\ra$ is of Nilsson type and thus that the asymptotic behavior does exist \cite{GaVV}. 
Abdesselam obtained estimations on the growth of spin-network evaluations, specially for generalized drum graphs, see \cite{Ab}.

Another interesting approach to the study of spin networks was proposed by Westbury who computed in \cite{We} the generating series of spin networks as follows (we refer to Section \ref{spin-network} for notation).
Let $R_X=\C[[X_{\alpha},\alpha\in A]]$ be the ring of formal variables associated to the angles of $\Gamma$. For any coloring $c$, we denote by $c_e,c_\alpha$ and $X^c$ respectively the color of the edge $e$, that of the angle $\alpha$ and the monomial $\prod_{\alpha\in A} X_{\alpha}^{c_{\alpha}}$. Then, writing $\la\la \Gamma, c\ra\ra= \la\Gamma,c\ra\frac{\prod_{e\in E}c_e!}{\prod_{\alpha\in A}c_{\alpha}!}$, we define the following formal series: $$Z(\Gamma)= \sum_{c} \la\la \Gamma,c\ra\ra X^c.$$
Given $\delta\subset \Gamma$, a subgraph which is a disjoint union of cycles, we denote by $c_{\delta}$ the coloring which associates to an angle 1 or 0 if $\delta$ contains the two edges forming the angle or not respectively. Let then $P_{\Gamma}= \sum\limits_{\delta\subset \Gamma}  X^{c_{\delta}}$.

\begin{theorem*}[Westbury, \cite{We} ] \label{theo:westbury}
If $\Gamma$ is a planar graph then $Z(\Gamma)=P_{\Gamma}^{-2}$.
\end{theorem*}
This was generalized to all trivalent graphs by Garoufalidis and Van der Veen \cite{GaVV}.
Here below we describe our results which may be presented in two independent parts.

\subsection{Generating series with holonomies}
In gauge theory it is natural to consider spin networks whose edges are decorated by a \emph{holonomy} $\psi$ in $\mathrm{SL}_2(\C)$: we shall denote it by $\la\Gamma,c,\psi\ra$ (see Definition \ref{def:sn}) and its renormalization by $\la\la \Gamma, c,\psi\ra\ra$ (defined as above). In the present paper we generalize Westbury's theorem to the case of a general graph $\Gamma$ equipped with any holonomy $\psi$ and we give a closed formula for the generating series $Z(\Gamma,\psi)=\sum_c \la\la\Gamma,c,\psi\ra\ra X^c.$

Fix a numbering of the vertices of $\Gamma$ compatible with the planar structure, see Figure \ref{graphe} and let $\Gamma'$ be the graph obtained from $\Gamma$ by blowing up vertices, i.e. replacing \raisebox{-0.1cm}{\includegraphics[width=.6cm]{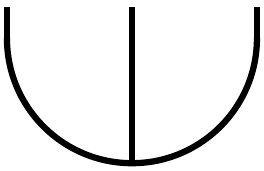}}$\to$\raisebox{-0.1cm}{\includegraphics[width=.6cm]{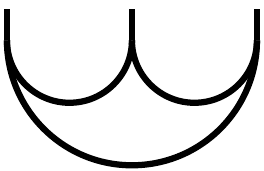}}.
Vertices of $\Gamma'$ are in 1-to-1 correspondence with the set $H$ of half-edges of $\Gamma$ and the holonomy may be seen as a map $\psi:H\to\mathrm{SL}_2(\C)$. 
Then, we set $F_h=\R^2$ for every half-edge and for any pair of half-edges $g,h$, we define $b_{g,h}:F_g\times F_h\to \C$ by the formula $b_{g,h}(z_g,w_g,z_h,w_h)=i(z_gw_h-z_hw_g)$.

We define the following $R_X$-valued quadratic form on $\bigoplus\limits_{h\in H} F_h:$
$$\mathcal{Q}(x)=2\sum_{\alpha:g\to h} b_{g,h}(\psi_g^{-1}x_g, \psi_h^{-1}x_h)X_{\alpha}+2\sum_{e:g\to h} b_{g,h}(x_g,x_h).$$
The expression $\alpha:g\to h$ means that $\alpha$ is an angle between the half-edges $g$ and $h$ such that the opposite vertex of $g$ is lower than the opposite vertex of $h$. In the same way, $e:g\to h$ means that the edge $e$ contains the half-edges $g$ and $h$ in such a way that the vertex contained in $g$ is lower than the vertex contained in $h$. Using Gaussian integrals, we prove the following result:

\begin{theorem}\label{thm1}
Given a planar graph $\Gamma$ with holonomy $\psi:H\to \mathrm{SL}_2(\C)$, the following formula holds $Z(\Gamma,\psi)=\det(\mathcal{Q})^{-1/2}$.
\end{theorem}
In this expression, the determinant is computed in the canonical basis. 
We remark that strictly speaking, the formula makes sense only for holonomies in SL$_2(\R)$ because of the indeterminacy in the square root. By analytic continuation, the formula holds in general. 
The non-planar case, can be easily treated as a corollary of Theorem \ref{thm1}: see Subsection \ref{sub:nonplanarcase}. As a corollary we observe that $Z(\Gamma,\psi)^{-2}$ is a polynomial with integer coefficients in the entries of $\psi$.
In Theorem \ref{teo:abelianwestbury}, we provide a combinatorial interpretation of our formula which extends Westbury's formula in the case when $\psi$ takes values in the diagonal subgroup of $\mathrm{SL}_2(\C)$.

\subsection{Asymptotics from integral formulas}
In his book \cite{wigner}, Wigner showed that the square of a $6j$-symbol may be computed by a simple integral formula over $4$ copies of SU$_2$. Barrett observed that this formula may be generalized to any graph \cite{barrett}. In Section \ref{integrals}, after recalling the integral formula, we compute the generating series of squares of spin networks:
\begin{theorem}
Let $(\Gamma,c,\psi)$ be a spin network equipped with a holonomy with values in $\mathrm{SL}_2(\C)$.
Define $$[\Gamma,c,\psi]=\frac{\la\Gamma,c,\psi \ra^2} {\prod\limits_{v:(e_1,e_2,e_3)} \la \Theta,c_{e_1},c_{e_2},c_{e_3}\ra}$$ where by $v:(e_1,e_2,e_3)$ we indicate the edges touching $v$ and $\la \Theta,c_{e_1},c_{e_2},c_{e_3}\ra$ is the value of the $\Theta$-graph colored by $c_{e_1},c_{e_2},c_{e_3}$. Considering the generating series  $W(\Gamma,\psi)=\sum_{c}[\Gamma,c,\psi]\ Y^{c_e}$,
the following holds in $R_Y=\C[Y_e,e\in E]$:
\begin{eqnarray*}
W(\Gamma,\psi)&=&\int_{G^V}\frac{\dd g}{\prod\limits_{e=(h_1,h_2)=(v,w)}\det(\psi_{h_2} g_w\psi_{h_2}^{-1}-\psi_{h_1} g_v\psi_{h_1}^{-1}Y_e)}
\end{eqnarray*}

\end{theorem}
Using Kirillov trace formula, the integral formulas may be transformed in order to apply the stationary phase approximation. This method was applied by \cite{barrett,freidel} to compute the asymptotical behavior of some spin networks but it faces some technical difficulties because of the existence of so-called ``degenerate configurations". 
In Section \ref{asymptotic} we propose a different transformation which allows us to treat uniformly all configurations corresponding to critical points.
We describe precisely the critical points of the integrand and compute the associated Hessian. Under suitable genericity hypotheses on $\Gamma$ and $c$ described in Section \ref{hypo}, we compute the dominating terms in the asymptotical development of $[\Gamma,kc,1]$ for general $\Gamma$.
Fix a trivalent graph $\Gamma$ and a coloring $c:E\to \N$. 

Let $I$ be the set of maps $P$ from oriented edges of $\Gamma$ to $S^2$ which satisfy the following relations:
\begin{itemize}
\item[-] Denoting by $-e$ the edge $e$ with opposite orientation, we have $P_{-e}=-P_e$.
\item[-] For all vertex $v$ with outgoing edges $e_1,e_2,e_3$ we have $\sum_i c_{e_i}P_{e_i}=0$.
\end{itemize}

Given $P\in I$ we define $r_P(\xi)=\sum_e c_e || P_e\times \xi||^2$ for $\xi\in \R^3$ and $q_P(\xi_v)=\sum_{e:(v,w)}c_e || P_e\times (\xi_v-\xi_w)||^2$ for $(\xi_v)\in \bigoplus_{v\in V}\R^3$. In these formulas, we denoted by $\times$ the bracket in the cross-product in $\R^3$.

Given a pair $(P,Q)$ of non-isometric elements of $I$, we define its \emph{phase function} $\tau$  in Subsection \ref{hypo} as a map from $E$ to $S^1/\{\pm 1\}$. Then we set  for $(\xi_v)\in \bigoplus_{v\in V}\R^3$:
$$q^{\kappa}_{P,Q}(\xi_v)=\sum_e c_e\big(\frac{\kappa^2\tau_e^2+1}{\kappa^2\tau_e^2-1}||Q_e\times(\xi_v-\xi_w)||^2+2i\la Q_e,\xi_v\times \xi_w\ra\big).$$
Then, given a quadratic form $q$ on $\mathbb{R}^n$, we denote by $\det'(q)$ the determinant of the restriction of the quadratic form to the orthogonal of the kernel of $q$, that is the product of all non-zero eigenvalues of the matrix of $q$. We also set $\det'(q_{P,Q})=\lim_{\kappa\to 1}(\kappa-1)^{-3}\det'(q^{\kappa}_{P,Q})$.

\begin{theorem}\label{asymptotic} Let $(\Gamma,c)$ be a colored graph satisfying the conditions of Subsection \ref{hypo}. Then denoting by $N$ the opposite of the Euler characteristic of $\Gamma$, one has 
\begin{eqnarray*}
[\Gamma,kc]&=&\frac{(2N)^{3/2}}{(\pi k^3)^{N-1}}\Big(
\sum_{P\in I}\frac{\det(r_P)^{1/2}}{\det'(q_{P})^{1/2}}\\
&&+\sum_{(P,Q)\in I^2/\pm 1, P\ne Q}2\textrm{Re}\Big(
\frac{i^{N}\det(r_P)^{1/2}e^{i\sum_e (kc_e+1)\theta_e}}
{\det'(q_{P,Q})^{1/2}\prod_e \sin(\theta_e)}\Big)+O(k^{-1})\Big)
\end{eqnarray*}
\end{theorem}
The above formula was numerically tested in the case of the tetrahedron and of course  its results are in line with the known asymptotics (\cite{roberts}, \cite{GaVV}, \cite{barrett}, \cite{charles}).  
The non-degeneracy conditions  of Subsection \ref{hypo} are equivalent to the non-vanishing of all determinants in the formula of Theorem \ref{asymptotic}.

\begin{itemize}
\item[-]
The quantity $\det(r_P)$ is zero only if the configuration $P$ is planar which can occur only for very special values of $c$.
\item[-]
The non-vanishing of $\det'(q_P)$ is equivalent to the infinitesimal rigidity of the configuration $P$. In particular, the hypothesis do not hold if the set $I$ is not discrete: this happens for instance for the regular cube or more generally, the cube whose edges are colored by the lengths of Bricard's flexible octahedron.
\item[-]
We do not have a geometric interpretation of the determinant $\det'(q_{P,Q})$ but in our numerical experiments on the spin-networks formed by the $1$-skeleton of a tetrahedron and of a triangular prism this determinant was non-zero. In any case, the conditions define a Zariski open set of configurations.
\end{itemize}
We believe that describing when the non-degeneracy conditions hold is a very difficult task as it contains the problem of the flexibility of polyhedra, a notoriously hard problem. Still, we expect that for planar graphs whose colors correspond to the lengths of a generic convex configuration of the dual graph, there is a simple geometric condition ensuring that the non-degeneracy conditions hold but this question is not addressed in this article.

\subsection{Some questions}
\begin{itemize}
\item[-] The generating series $Z(\Gamma,\psi)^2$ has coefficients which are integral polynomials in the entries of $\psi$. Is it still true for $Z(\Gamma,\psi)$ as suggested by the abelian case?
\item[-] Is there a combinatorial interpretation of $Z(\Gamma,\psi)$ for general $\psi$?
\item[-] Find a direct relation between the series $Z(\Gamma,\psi)$ and $W(\Gamma,\psi)$?
\item[-] Find sufficient conditions on $\Gamma$ and $c$ which ensure that the non-degeneracy hypothesis hold, in particular in the case of convex polyhedra.
\item[-] Interpret geometrically the leading terms in the asymptotic formula of Theorem \ref{asymptotic}.
\end{itemize}
\subsection{Acknoweldgements}
The authors would like to thank Fr\'ed\'eric Faure for his comments on coherent states, leading us to Lemma \ref{coh}.
This work was supported by the French ANR project ANR-08-JCJC-0114-01.

\section{Equivalent descriptions of spin network evaluations}\label{spin-network}

Let $\Gamma$ be a trivalent graph with possibly loops and multiple edges. We denote by $E$ the set of edges, $V$ the set of vertices and we divide each edge $e\in E$ into two subarcs called \emph{half-edges}. We will then let $H$ be the set of half-edges commonly described as pairs $(e,v)$ where $e$ is an edge and $v$ is an end of $e$. We shall moreover orient a priori each half-edge $h=(e,v)$ so that it goes out of the vertex $v$. We assume that for each vertex the set of half-edges incoming to that vertex have a cyclic order. We call \emph{angle} a pair of half-edges touching the same vertex and denote by $A$ the set of angles. In all the article, we will suppose that $\Gamma$ is connected and contains at least a vertex; we will denote by $N$ the opposite of the Euler characteristic of $\Gamma$: we have $\#V=2N, \#E=3N$ and $\#H=\#A=6N$.

We will denote by \emph{admissible coloring} a map $c:E\to \N$ satisfying the triangle conditions:
\begin{equation}\tag{T}
\forall v:(i,j,k),\,\, c_i+c_j+c_k\in 2\N\textrm{ and } c_i\le c_j+c_k
\end{equation}
For convenience, we wrote $v\!:\!(i,j,k)$ meaning that the edges $i,j,k$ are incoming at $v$ with that cyclic order.
Associated to each coloring $c$ there is an \emph{internal coloring}, i.e. a map from $A$ to $\mathbb{N}$ also denoted by $c$ and defined as follows: if $\alpha$ is the angle between edges $i$ and $j$ around $v$ then we set $c_{\alpha}=\frac{c_i+c_j-c_k}{2}$. We remark than one can recover the original coloring from the internal coloring: to avoid confusion, we will denote edges with latin letters and angles with greek ones.

Let $(V,\omega)$ be a complex symplectic vector space of rank 2 and SL($V$) be its symmetry group. A \emph{discrete connection} is a map $\hol:H\to$ SL$(V)$; we define the \emph{gauge group} as the group of maps $\{g:V\amalg E\to \textrm{SL}(V)\}$. An element $g$ of the gauge group acts on a discrete connection $\hol$ as follows: for each $h=(e,v)\in H$, ($g\cdot \hol)_h= g_e\hol_h g_v^{-1}$. Two discrete connections $\hol_1,\hol_2: E\to \textrm{SL}(V)$ are said to be \emph{gauge equivalent} if they are in the same orbit of the gauge group. 

\begin{definition}[Holonomy]
\begin{itemize}
\item[-]
A holonomy on $\Gamma$ is an equivalence class of discrete connections $\hol$ on $\Gamma$. \item[-]
The trivial holonomy (denoted by 1) is the class of the constant discrete connection defined by $\hol_h=1\in\mathrm{SL}(V)$ for all half-edges $h\in H$.
\item[-] Let $\gamma$ be an oriented path in $\Gamma$ described as a sequence of half-edges $h_1,\ldots,h_n$. We define the holonomy $\hol(\gamma)$ of $\hol$ along $\gamma$ to be the product $\hol_{h_n}^{\epsilon_n}\cdots \hol_{h_1}^{\epsilon_1}$ where $\epsilon_i$ is 1 or -1 if the half-edge is oriented coherently with $\gamma$ or not. If $\hol_1$ and $\hol_2$ are gauge equivalent and $\gamma$ is closed then $\hol_1(\gamma)$ and $\hol_2(\gamma)$ are conjugate.
\end{itemize} 
\end{definition}

In the following section, we give two equivalent descriptions of the spin network evaluation which consists in associating to a triple $(\Gamma,c,\hol)$ a complex number $\la \Gamma,c,\hol\ra$. We will refer to these descriptions as the abstract and computational descriptions.

\subsection{Abstract description}\label{sub:abstractdef}
Let $(V,\omega)$ be a complex symplectic vector space of rank 2. We will consider that $V$ is an odd superspace. We refer to Appendix \label{super} for a basic review of supersymmetry adapted to our purposes.

The symplectic form is considered as a supersymmetric map $\omega:V\otimes V\to \C$ in the sense that $\omega\circ\phi_{(12)}=\omega$. For any integer $n\in \N$, $\omega$ induces a map $\omega^{\otimes n}:V^{\otimes n}\otimes V^{\otimes n}\to \C$ defined by the formula 
$$\omega^{\otimes n}(v_1\otimes\cdots\otimes v_n,w_1\otimes\cdots\otimes w_n)=\prod_{i=1}^n \omega(v_{n+1-i},w_i).$$

Denote by $V_n$ the subspace of $V^{\otimes n}$ consisting in anti-supersymmetric tensors (or, equivalently, symmetric in the standard sense); the form $\omega^{\otimes n}$ restricts to a supersymmetric form $\omega_n:V_n\otimes V_n\to \C$. The vector space $V_n$ (of parity $n$) is the $(n+1)$-dimensional irreducible representation of the group SL($V$): we will sometimes use the notation $\rho_n:\textrm{SL}(V)\to\textrm{End}(V_n)$. Denote by $\omega^{-1}$ the unique element of $V\otimes V$ such that the contraction of the two middle terms in $\omega\otimes \omega^{-1}$ is the identity of $V$. 

Let $a,b,c$ be three integers. It is well known that the set of SL($V$)-invariant elements in $V_a\otimes V_b\otimes V_c$ is 1-dimensional if $a,b,c$ satisfy the triangle conditions (T) unless it is 0.
 One can find an explicit generator $\epsilon_{a,b,c}\in V_a\otimes V_b\otimes V_c$ given by 
 the symmetrization of the element
 $$\omega_{a,b,c}=(\omega^{-1})^{\otimes (a+b-c)/2}\otimes (\omega^{-1})^{\otimes (b+c-a)/2}\otimes (\omega^{-1})^{\otimes (a+c-b)/2}$$
 where the supersymmetric tensor product is reordered as in Figure \ref{inv}. Remark that the sign of the permutation which reorders the $a+b+c$ factors in this tensor product is +1.
 \begin{figure}
\centering
  \def\svgwidth{\columnwidth}
 \executeiffilenewer{omegas.svg}{omegas.pdf}%
 {inkscape -z -D --file=omegas.svg %
 --export-pdf=omegas.pdf --export-latex}%
 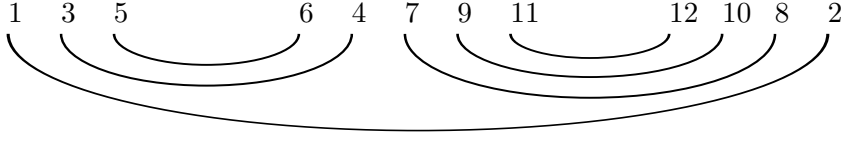%

  \caption{Invariant tensor at vertices with $a=3,b=5,c=4$.\label{inv}}
\end{figure}

 Then we define $\epsilon_{a,b,c}=(\Pi_a\otimes \Pi_b\otimes \Pi_c)\omega_{a,b,c}$ where $\Pi_a$ is the anti-supersymmetrization (i.e. standard symmetrization) map projecting $V^{\otimes a}$ onto $V_a$.

One may check that this element is supersymmetric as if we permute cyclically $a,b$ and $c$ to the left the result is multiplied by $(-1)^a=(-1)^{a(b+c)}$ which is the sign of the cycle $V_a\otimes V_b\otimes V_c\to V_b\otimes V_c\otimes V_a$ according to the supersymmetric rule.

\begin{definition}[Spin Network]
Let $\Gamma$ be a trivalent graph and $c:E\to\N$ be an admissible coloring. Then, $\la \Gamma,c,1\ra$ is the result of the supersymmetric contraction of 
$$\la\bigotimes\limits_{e\in E} \omega_{c_e},\bigotimes\limits_{v\in V,v:(i,j,k)}\epsilon_{c_i,c_j,c_k}\ra.$$ 
\end{definition}
By supersymmetric contraction, we mean that we reorder the tensors on the right hand side according to the sign rule so that factors corresponding to the same edge are consecutive, and then contract with the maps $\omega_n$.

Given a discrete connection $\hol: H\to \textrm{SL}(V)$, let us define the following endomorphism $\textrm{Hol}_{\psi}$ by the formula: $$\bigotimes\limits_{v\in V,v:(i,j,k)} \rho_{c_i}(\hol_{i,v})\otimes\rho_{c_j}(\hol_{j,v})\otimes\rho_{c_k}(\hol_{k,v})\in\textrm{End}\left(\bigotimes\limits_{v\in V,v:(i,j,k)} V_{c_i}\otimes V_{c_j}\otimes V_{c_k}\right)$$

\begin{definition}[Spin Network with holonomy]\label{def:sn}
Let $\Gamma$ be a trivalent graph, $\hol:H\to \textrm{SL}(V)$ be a discrete connection and $c:E\to\N$ be an admissible coloring. Then, $\la \Gamma, c,\hol\ra$ is the result of the following supersymmetric contraction:

$$\la \bigotimes\limits_{e\in E} \omega_{c_e},\textrm{Hol}_{\hol}\big(\bigotimes\limits_{v\in V,v:(i,j,k)}\epsilon_{c_i,c_j,c_k}\big)\ra
.$$ 
\end{definition}
One can check directly that this definition does not depend on the gauge equivalence class of $\psi$.
\subsection{Computational description}
The computational description follows directly from the abstract one, taking care of the signs.

Set $V=\C^2$ and $\omega(v,w)=\det(v,w)$. Then $V^*=\C^2$ and we denote by $z$ and $w$ the corresponding coordinates. In this way, $\omega^{-1}\in V\otimes V$ corresponds to the linear polynomial on $V^*\times V^*$ which in coordinates $(z_1,w_1,z_2,w_2)$ reads as $z_2 w_1-z_1w_2$. 

The form $\omega_n:V_n\otimes V_n\to \C$ is a linear form on the space of homogeneous polynomials of bidegree $(n,n)$ on $V^*\times V^*$. Here we identify homogeneous polynomials and symmetric tensors via:
$$z^iw^{n-i}\to \frac{1}{n!}\sum_{\sigma\in \mathfrak{S}_n} X_{\sigma_1}\otimes \cdots \otimes X_{\sigma_n}$$ where $X_k=z$ if $k\leq i$ and $w$ otherwise).
One checks directly that in coordinates, $\omega_n$ is expressed by:
$$\omega_n=\frac{1}{n!^2}(\frac{\partial}{\partial z_1}\frac{\partial}{\partial w_2}-\frac{\partial}{\partial z_2}\frac{\partial}{\partial w_1})^n$$

In the same way, the element $\epsilon_{a,b,c}\in V_a\otimes V_b\otimes V_c$ corresponds to a polynomial on $V^*\times V^*\times V^*$ which in coordinates $(z_1,w_1,z_2,w_2,z_3,w_3)$ reads as 
$$\epsilon_{a,b,c}=(z_2 w_1-z_1w_2)^{\frac{a+b-c}{2}}(z_3 w_2-z_2 w_3)^{\frac{b+c-a}{2}}(z_3 w_1-z_1w_3)^{\frac{a+c-b}{2}}.$$

\begin{figure}
\centering
  \def\svgwidth{\columnwidth}
 \executeiffilenewer{tetraedre.svg}{tetraedre.pdf}%
 {inkscape -z -D --file=tetraedre.svg %
 --export-pdf=tetraedre.pdf --export-latex}%
 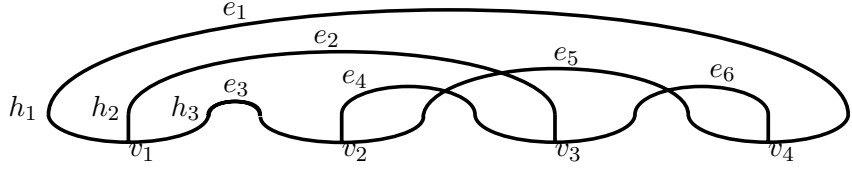%

  \caption{Planar presentation of the tetrahedron\label{graphe}}
\end{figure}

Finally, let us compute the sign needed to perform the final contraction. Suppose that the trivalent graph is presented as in Figure \ref{graphe}: this presentation induces an ordering of the set of half-edges $H$. For any edge $e\in E$, let $e_1$ and $e_2$ be the respective left and right half-edges of $e$ and for any vertex $v$ in $V$, let $v_1,v_2,v_3$ be the three half-edges incoming to $v$ in increasing order.
Writing in line the tensor product of invariant elements $\epsilon_{a,b,c}$ give a big tensor which has to be reordered in such a way that half edges are packed in pairs. Let $C$ be the set of pairs of crossing edges: the sign of this permutation is $(-1)^{\sum_{\{e,e'\}\in X} c_e c_{e'}}$.

We compute then:
\begin{eqnarray}\label{eq:DefSpinNetwork}
\lefteqn{\la \Gamma,c,1\ra=(-1)^{\sum_{\{e,e'\}\in X} c_e c_{e'}}\prod_{e} \frac{1}{c_e!^2}\left(\frac{\partial}{\partial z_{e_1}}\frac{\partial}{\partial w_{e_2}}-\frac{\partial}{\partial z_{e_2}}\frac{\partial}{\partial w_{e_1}}\right)^{c_e}}\nonumber\\
& &\prod_{v}(z_{v_2} w_{v_1}-z_{v_1}w_{v_2})^{\frac{c_{v_1}+c_{v_2}-c_{v_3}}{2}}(z_{v_3} w_{v_2}-z_{v_2} w_{v_3})^{\frac{c_{v_2}+c_{v_3}-c_{v_1}}{2}}\\
& &(z_{v_3} w_{v_1}-z_{v_1}w_{v_3})^{\frac{c_{v_1}+c_{v_3}-c_{v_2}}{2}}.\nonumber
\end{eqnarray}

In order to use the machinery of Gaussian integration, it will be more comfortable to introduce $i=\sqrt{-1}$ in our  formulas. The formula is unchanged if we replace $\omega_n$ and $\epsilon_{a,b,c}$ by $\ba{\omega}_n$ and $\ba{\epsilon}_{a,b,c}$ where we set 
$$\ba{\omega}_n=\frac{i^n}{n!^2}(\frac{\partial}{\partial z_1}\frac{\partial}{\partial w_2}-\frac{\partial}{\partial z_2}\frac{\partial}{\partial w_1})^n$$
$$\ba{\epsilon}_{a,b,c}=i^{\frac{a+b+c}{2}}(z_1 w_2-z_2w_1)^{\frac{a+b-c}{2}}(z_2 w_3-z_3 w_2)^{\frac{b+c-a}{2}}(z_1 w_3-z_3w_1)^{\frac{a+c-b}{2}}.$$

The formula for $\la \Gamma,c,\psi\ra$ is obtained from Equation \eqref{eq:DefSpinNetwork} by replacing the variables $(z_i,w_i)$ at the vertex $v$ with $\psi_{i,v}^{-1}(z_i,w_i)$.

\section{Generating series via gaussian integrals}

In Subsection \ref{sub:proofseries} we will use gaussian integration to prove Theorem \ref{thm1} in the case of a planar graph $\Gamma$. This is a generalization of Westbury's result to the case when graphs are equipped with holonomies.
In Subsections \ref{sub:abeliancase} we will provide a topological interpretation of the theorem when the holonomy is diagonal and recover Westbury's original result (always in the case of planar graphs). In Subsection \ref{sub:nonplanarcase} we will extend Theorem \ref{thm1} to the case of non-planar graphs. 

\subsection{Computing the generating series of a spin network}\label{sub:proofseries}
Let $\Gamma$ be a planar trivalent graph that we present as in Figure \ref{graphe}. It may be easily seen that the sign correction due to crossings in Equation \eqref{eq:DefSpinNetwork} is $1$.
Let $F_h$ be a copy of $\R^2$ associated to each half-edge $h$, we will always consider the standard density on these spaces and then omit it in the notation, and let $F=\bigoplus_h F_h$. For two half-edges $g,h$ we set $b_{g,h}:F_{g}\times F_{h}\to \C$ as being given in coordinates by the expression $i(z_gw_h-z_hw_g)$ and $b^{-1}_{g,h}:F^*_{g}\times F^*_{h}\to \C$ as $-i(\frac{\partial}{\partial z_g}\frac{\partial}{\partial w_h}-\frac{\partial}{\partial z_h}\frac{\partial}{\partial w_g})$. 

Define quadratic forms $P$ on $F^*$ and $Q$ on $F$ respectively by $P=-2\sum\limits_{e:g\to h}b^{-1}_{g,h}$ and $Q=2\sum\limits_{\alpha:g\to h} X_{\alpha}b_{g,h}$. The notation $\alpha:g\to h$ and $e:g\to h$ means that $\alpha$ and $e$ are composed of the half-edges $g,h$ which appear in that order from left to right in Figure \ref{graphe}. For the moment, $X_{\alpha}$ should be interpreted as a real parameter.  
Note that the quadratic form $P^{-1}$ on $F$ is expressed by $P^{-1}=2\sum\limits_{e:g\to h}b_{g,h}$.

Consider $P$ as a differential operator $P^{op}$ on $C^{\infty}(F,\C)$ and develop $(\exp(\frac{1}{2}P)^{op} \exp(\frac{1}{2}Q))|_0$ (see Appendix \ref{gaussian}). Collecting monomials in the variables $X_{\alpha}$, by Equation \eqref{eq:DefSpinNetwork} one sees that the coefficient of $\prod_{\alpha} X_\alpha^{c_\alpha}$ is $\la\la\Gamma,c\ra\ra$ (defined in the Introduction). Therefore we have: 
$$Z(\Gamma,1)=\big(\exp(\frac{1}{2}P)^{op} \exp(\frac{1}{2}Q)\big)|_0.$$

Consider a deformation $Q_{\epsilon}$ of $Q$ which is non-degenerate and has positive real part. For concreteness, we can pick $Q_0=\sum_h (z_h^2+w_h^2)$ and set $Q_{\epsilon}=Q+\epsilon Q_0$. We define $Z_{\epsilon}=\big(\exp(\frac{1}{2}P)^{op} \exp(\frac{1}{2}Q_{\epsilon})\big)|_0$ so that $Z(\Gamma,1)=\lim_{\epsilon\to 0} Z_{\epsilon}$.
Then replacing $Q$ by $Q_{\epsilon}^{-1}$ and $\frac{1}{2}P$ by $\exp(\frac{1}{2}P)$, Formula \eqref{gauss2} gives

$$Z_{\epsilon}=(2\pi)^{-n/2}\det(Q_{\epsilon}^{-1})^{1/2}\int_{F^*}\exp(\frac{1}{2}P(x)-\frac{1}{2}Q_{\epsilon}^{-1}(x))\dd x.$$

We apply now Formula \eqref{gauss} to the integral, noting that the quadratic form $Q_{\epsilon}^{-1}-P$ is still non-degenerate and has positive real part. Hence, we have 
$$Z_{\epsilon}^2=\frac{\det(Q_{\epsilon}^{-1})}{\det(Q_{\epsilon}^{-1}-P)}=\det(Q_0-Q_{\epsilon}P)^{-1}$$
Letting $\epsilon$ go to 0, we find that $Z(\Gamma,1)=\det(Q_0-QP)^{-1/2}$.

Suppose that $\hol$ is represented by a discrete connection on $\Gamma$ with values in SL$_2(\R$). By Formula  \eqref{eq:DefSpinNetwork}, we know how to adapt the construction: for each angle $\alpha$ connecting two half-edges $g$ and $h$, we need to replace $b_{g,h}:F_g\times F_h\to \C$ by $b_{g,h}(\hol_g^{-1}x_g,\hol_h^{-1}x_h)$. We denote by $Q_{\hol}$ the resulting quadratic form. By the assumption that $\hol$ lives in SL$_2(\R)$, $P$ takes again only imaginary values. 
Hence, the argument above repeats exactly and we obtain $Z(\Gamma,\hol)=\det(Q_0-Q_{\hol}P)^{-1/2}$. The general case, that is, for $\hol$ taking values in SL$_2(\C)$ follows by analytic continuation.

One can simplify this formula by remarking that the matrix of $P$ in the canonical basis satisfy $P^{-1}=-P$ and moreover, we have $\det(P)=1$. We obtain the formula of Theorem \ref{thm1}:

\begin{theorem*}[\ref{thm1}]
$$Z(\Gamma,\hol)=\det(P+Q_{\psi})^{-1/2}.$$
\end{theorem*}
\subsection{A generalization of Westbury's theorem to the case of diagonal holonomies}
\label{sub:abeliancase}

Suppose that the holonomy $\psi$ can be represented by a connection with values in the subgroup D$\subset$SL$_2(\R)$ of diagonal matrices, we introduce a map $t:H\to \R^*$ such that for all $h$, one has $\psi_h=\begin{pmatrix}t_h&0\\0&t_h^{-1}\end{pmatrix}$ in the basis $(z_h,w_h)$. In this case, we can extend Westbury's theorem \ref{theo:westbury} as follows.

Let $C(\Gamma)$ be the set of all oriented curves immersed in $\Gamma$ which pass over an edge of $\Gamma$ either $0,1$ or $2$ times, in the latter case with opposite orientations. 
Given $\gamma\in C(\Gamma)$ we denote by ${\rm cr}(\gamma)$ the number of crossings modulo 2 of the corresponding immersion.

Let $X^\gamma=\prod_{\alpha\subset \gamma} X_{\alpha}$ (hence each angle may appear 0, 1 or 2 times) and $\tr(\gamma)=\prod_{h\subset \gamma} t_h^{\epsilon(\gamma,h)}$ ($\epsilon(\gamma,h)$ being $1$ if $\gamma$ crosses $h$  in the positive direction, $-1$ otherwise). 
The following result generalizes Westbury's result to the case of holonomies with values in diagonal matrices: 
\begin{theorem}\label{teo:abelianwestbury}
Let $\Gamma$ be a planar graph equipped with an abelian holonomy $\psi$. Then: $$Z(\Gamma, \psi)=\big(\sum_{\gamma\in C(\Gamma)} (-1)^{{\rm cr}(\gamma)} \tr(\gamma)X^{\gamma}(c)\big)^{-1}$$  
\end{theorem}
Theorem \ref{teo:abelianwestbury} descends directly from Corollary \ref{cor:determinants} and Proposition \ref{prop:countingcurves} of Appendix \ref{sub:determinants} where we also show how to recover Westbury's theorem from it.

\subsection{The non-planar case}\label{sub:nonplanarcase}
Let now $\Gamma$ be a non-planar graph and let us fix a diagram of $\Gamma$ as in Figure \ref{graphe} containing crossings $x_1,\ldots x_k$.
By the supersymmetric rules, for each coloring $c$ of $\Gamma$, a crossing $x_i$ between edges $e_1$ and $e_2$ induces a factor $(-1)^{c_{e_1}c_{e_2}}$ which we denote by $s(x_i,c)$. The previous results allow us to compute the ``wrong" generating series for $\Gamma$ where the signs coming from crossings are not taken into account, that is we can compute: $W(\Gamma,\psi)=\sum_{c} \la\la\Gamma,c,\psi\ra\ra X^c\prod_{x_i} s(x_i,c)$.
To fix these signs, use the identity $$(-1)^{ab}=\frac{1}{2}(1+(-1)^a+(-1)^b-(-1)^{a+b}),\ \forall a,b\in \mathbb{Z}.$$ More explicitly, for each edge $e$ let $\alpha_e,\beta_e$ be the leftmost angles formed by $e$ and let ${\rm Op}_{e}:\mathbb{C}[[X]]\to\mathbb{C}[[X]]$ be the automorphism that changes the signs of $X_{\alpha_e}$ and $X_{\beta_e}$. Then, for a crossing $x$ between edges $e_1$ and $e_2$ define $S_x:\mathbb{C}[[X]]\to\mathbb{C}[[X]]$ as $S_{x}=\frac{1}{2}(\id+{\rm Op}_{e_1}+{\rm Op}_{e_2}-{\rm Op}_{e_1}\circ{\rm Op}_{e_2})$.  Then we recover $Z(\Gamma,\psi)=\sum_{c} \la\la\Gamma,c,\psi\ra\ra X^c$ as 
$$Z(\Gamma,\psi)=S_{x_1}\circ\cdots\circ S_{x_k}(W(\Gamma,\psi)).$$
In particular, when $\psi$ is trivial this recovers Garoufalidis and Van der Veen's extension to non-planar graphs of Westbury's theorem \cite{GaVV}.

\section{An integral formula for the square of spin network evaluations and its analysis}\label{integrals}
In this section we deduce an integral formula for the square of the spin-network evaluation. In the case of the tetrahedron, this formula was already known to Wigner (see \cite{wigner}). It was generalized by Barrett in \cite{barrett} to any graph. For questions of clarity and normalizations, we provide a proof of it, then we derive a formula for the corresponding generating series.
From now on we will no longer work in the super-symmetric category: being interested in squares of spin-network evaluations, sign matters are irrelevant.

\subsection{Derivation of the integral formula}
Let $(V,\omega,h)$ be a symplectic complex and hermitian vector space of rank 2. Both structures are supposed to be compatible in the sense that there is a hermitian basis $(e_1,e_2)$ of $V$ such that $\omega(e_1,e_2)=1$. We also suppose that the hermitian product is anti-linear on the left and denote by SU($V$) the symmetry group of the whole structure.  Naturally, $V^{\otimes n}$ has a hermitian structure as its subspace $V_n$. Denote also by $h$ the hermitian product induced on the spaces $V_n$ and their tensor products.

We would like to replace contractions with the forms $\omega_n$ by scalar products with some vectors. 
 The compatibility of the hermitian structure and the symplectic form give immediately the formula: $h(\omega^{-1},v\otimes w)=\omega(w,v)$ or $h(\varpi,v\otimes w)=\omega(v,w)$ where we set $\varpi=-\omega^{-1}$. Let $\varpi_n=(\Pi_n\otimes \Pi_n)\bigotimes_{i=1}^n\varpi_{i,2n+1-i}$, where $\Pi_n:V^{\otimes n}\to V^{\otimes n}$ is the projector on symmetric tensors in the non-supersymmetric sense. 
After some computation, we obtain for any $v,w\in V_n$, $h(\varpi_n,v\otimes w)=\omega_n(v,w)$.
 
Then we can compute the spin network evaluation using scalar product instead of contraction, more precisely we have up to sign:
\begin{equation}\label{spe-hermitian}
\la\Gamma,c\ra=h\left( \bigotimes\limits_{e\in E} \varpi_{c_e},\bigotimes\limits_{v\in V,v:(i,j,k)}\epsilon_{c_i,c_j,c_k}\right).
\end{equation}
Denote by $G$ the group SU($V$) and for any integer $n$, we denote by $\rho_n:G\to \text{U}(V_n)$ the induced representation of $G$.  We also write $\dd g$ for the Haar measure on $G$ satisfying $\int_G\dd g=1$.

Consider the element $$\boP=\int_G\rho_n(g)\otimes\rho_n(g)\dd g\in \text{End}(V_n\otimes V_n).$$ We check directly that $\boP$ satisfies $\boP\circ \boP=\boP$ and $\boP^*=\boP$. Moreover, the image of $\boP$ is the one dimensional space $(V_n\otimes V_n)^G$. Hence, $\boP$ is the orthogonal projector on $\C\varpi_n$. Using the formula $h(\varpi_n,v\otimes w)=\la\omega_n,v\otimes w\ra$ we have $h(\varpi_n,\varpi_n)=\la \omega_n,\varpi_n\ra$. But in the contraction of $\omega_n\otimes \varpi_n$, one can first contract the middle terms. One find $n$ times the evaluation of $\omega\otimes\varpi$ which is equal to $-\id_{V^*}$. Then, the middle contraction is $(-1)^n\id_{V_n^*}$. The final contraction computes the super-trace. Hence, one find $h(\varpi_n,\varpi_n)=n+1$.
Finally, we obtained that for any $v\in V_n\otimes V_n$, $\boP(v)=\frac{1}{n+1} h(\varpi^n,v)\varpi^n$.

In the same spirit, consider for any triple $(a,b,c)$ satisfying the triangular relations the element 
$$\boE=\int_G \rho_a(g)\otimes\rho_b(g)\otimes\rho_c(g)\dd g \in \text{End}(V_a\otimes V_b\otimes V_c).$$
The same reasoning as above shows that for any $v\in V_a\otimes V_b\otimes V_c$ one has $\boE(v)=\frac{h(\epsilon_{a,b,c},v)}{\la a,b,c\ra}v$ where we set $\la a,b,c\ra=h(\epsilon_{a,b,c},\epsilon_{a,b,c})$.
\begin{lemma}
$$\la a,b,c\ra=|\la\Theta,(a,b,c)\ra|=\frac{(\frac{a+b+c}{2}+1)!(\frac{a+b-c}{2})!(\frac{a-b+c}{2})!(\frac{-a+b+c}{2})!}{a!b!c!}$$
\end{lemma}
\begin{proof} A proof of it may be found for instance in \cite{bl} or computed from the generating series given by Theorem \cite{We} for the theta graph, taking care of the renormalization from $\la\la\Theta,(a,b,c)\ra\ra$ to $\la \Theta,(a,b,c)\ra$.\end{proof}

In equation \eqref{spe-hermitian}, replace $\epsilon$ by $\boE$ and $\varpi$ by $\boP$ and consider the scalar product
$h\left(\bigotimes_{e\in E} \boP_e,\bigotimes_{v\in V} \boE_v\right)$:
using the fact that all matrices involved are self-adjoint, one sees that it equals $\tr( \bigotimes_{e\in E} \boP_e\circ\bigotimes_{v\in V} \boE_v)$. 
For any edge $e$ colored by $c_e$, let $\la e\ra=c_e+1$ and $\boP_e$ be the projector acting on $V_{c_e}\otimes V_{c_e}$. Similarly, for any vertex $v$ whose incoming edges are $i,j,k$, we set $\la v\ra=\la c_i,c_j,c_k\ra$ and denote by $\boE_v$ the projector acting on $V_{c_i}\otimes V_{c_j}\otimes V_{c_k}$.
Let $[\Gamma,c]=\prod_{e}\la e\ra \tr( \bigotimes_{e\in E} \boP_e\circ\bigotimes_{v\in V} \boE_v)=\prod_{e}\la e\ra h\left(\bigotimes_{e\in E} \boP_e,\bigotimes_{v\in V} \boE_v\right)$.
  
 The first term in the scalar product defining $[\Gamma,c]$ is the orthogonal projector on $\bigotimes_{e}\varpi^{c_e}$ and the second term is the orthogonal projector on $\bigotimes_{v}\epsilon_v$ hence, we have:
 
\begin{equation}\label{relation}
[\Gamma,c]=\frac{|\la\Gamma,c\ra |^2}{\prod_v\la v\ra}.
\end{equation}
 Here we use the fact that $h$ induces a hermitian form on $\rm{End}(W)=W\otimes W^*
 $ and that the scalar product of projectors $P_u$ and $P_v$ on unit vectors $u$ and $v$ is $|\la u,v\ra|^2$.
 On the other hand, we can invert hermitian product and integration finding 
$$
[\Gamma,c]=\prod_e\la e\ra\int_{G^{V\amalg E}}\!h(\bigotimes_{e}\rho_{c_e}(g_e)\otimes\rho_{c_e}(g_e),\!\bigotimes_{v: (i,j,k)} \rho_{c_i}(g_v)\otimes\rho_{c_j}(g_v)\otimes \rho_{c_k}(g_v))\dd g$$
$$[\Gamma,c]=\prod_e\la e\ra\int_{G^{V\amalg E}}\prod_{h=(v,e)}\tr_{c_e}(g_v g_e)\dd g=\int_{G^V}\prod_{e:(v,w)}\tr_{c_e}(g_v g_w^{-1})\dd g.$$ 
 Remark that the last equation holds even if $c$ is a non admissible coloring on $\Gamma$ (in which case both sides of the equations are $0$).
 In order to prove the last equality, we used the following lemma:
 \begin{lemma}\label{prodtrace}
 For any integer $c$ and all elements $A,B\in G$ one has $$\int_G \tr_c(Ag)\tr_c(Bg)\dd g=\frac{1}{c+1}\tr_c(AB^{-1}).$$
 \end{lemma}
 \begin{proof}
 We have 
 \begin{eqnarray*}
 \int_G \tr_c(Ag)\tr_c(Bg)\dd g&=&\int_G\tr_c(AB^{-1}g)\tr_c(g)\dd g\\
 &=&\tr\rho_c\left(AB^{-1}\int_G\rho_c(g)\tr_c(g)\dd g\right).
 \end{eqnarray*}
 Writing $U=\int_G\rho_c(g)\tr_c(g)\dd g$ one obtains that $U$ commutes with $G$ hence $U$ is proportional to the identity.
 Moreover, $\tr(U)=\int_G\tr_c(g)^2\dd g=1$ hence $U=\frac{1}{c+1}\id_{V_c}$ which proves the lemma.
 \end{proof}
 
 As a summary we obtained the following formula, noting that $\la \Gamma,c\ra$ is real:
 $$\la\Gamma,c\ra^2=\prod_v\la v\ra \int_{G^V}\prod_{e:(v,w)}\tr_{c_e}(g_v g_w^{-1})\dd g=\prod_v\la v\ra [\Gamma,c]$$
 Notice that we need to orient arbitrarily the edges of $\Gamma$ to write this formula: this detail will be important in the next section.
 
Let $R=\C[[Y_e,e\in E]]$ be the ring of formal variables associated to the edges. To any coloring $c$, we denote by $Y^c$ the monomial $\prod_{e\in E}Y_e^{c_e}$. 
 For any $A$ in $G$, the following identity holds: 
 $$\sum_{n\in \N} \tr_n(A)Y^n=\frac{1}{\det(\id-YA)}.$$
 
Consider then $W(\Gamma)=\sum_{c} [\Gamma,c] Y^c$ where the sum is taken also over non-admissible colorings by setting $[\Gamma,c]=0$ in that case.
We deduce the following formula:
\begin{equation}\label{serie-integrale}
W(\Gamma)= \sum_{c \textrm{ admissible}} \frac{\la\Gamma,c\ra^2 Y^c} {\prod_{v} \la v\ra}
 =\int_{G^V}\frac{\dd g}{\prod\limits_{e:(v,w)}\det(\id -Y_e g_v g_w^{-1})}
 \end{equation}

\begin{example}
Consider the case of the graph $\smalltheta{ .5}{0.15}$. Then for any coloring $c$ given by three colors $a,b,c$, one has by construction $\la\Theta,a,b,c\ra=\la a,b,c\ra$ and hence the generating series \eqref{serie-integrale} is simply 
$$\sum_{c \textrm{, admissible}}Y^c=\frac{1}{(1-Y_1Y_2)(1-Y_2Y_3)(1-Y_1Y_3)}=\int_G\frac{\dd g}{\prod_{i=1}^3\det(\id-Y_ig)}.$$
 On the other hand, the coefficient of $Y_1^aY_2^bY_3^c$ in the above integral is $\int_G \tr_a(g)\tr_b(g)\tr_c(g)\dd g$. Orthonormality of the characters in $L^2(G,\dd g)$ and the Clebsch-Gordan rules imply that this integral is 1 if $a,b,c$ are admissible and 0 if not.
\end{example}

We can generalize the formula to the holonomy case: fix a holonomy $\psi:H\to G$. Then, one has to replace at each vertex $v$ the tensor $\epsilon_{c_i,c_j,c_k}$ with $\textrm{Hol}_{\psi}(\epsilon_{c_i,c_j,c_k})=\rho_{c_i}(\psi_{i,v})\otimes\rho_{c_j}(\psi_{j,v})\otimes\rho_{c_k}(\psi_{k,v}) \epsilon_{c_i,c_j,c_k}$. Then in the computation above, we need to replace the projector $\boE$ with the projector on $\textrm{Hol}_{\psi}(\epsilon_{c_i,c_j,c_k})$ that is $\textrm{Hol}_{\psi}\boE\textrm{Hol}_{\psi}^{-1}$. 
Therefore we have:
\begin{eqnarray*}
[\Gamma,c,\psi]&=&\prod_e\la e\ra\int_{G^{V\amalg E}}\prod_{h=(v,e)}\tr_{c_e}(g_e \psi_h g_v \psi_h^{-1})\dd g \\
&=&\int_{G^V}\prod_{e:(v,w)}\tr_{c_e}(\psi_{e,v}g_v\psi_{e,v}^{-1}\psi_{e,w}g_w^{-1}\psi_{e,w}^{-1})\dd g
\end{eqnarray*}
Since as before the above equalities still hold  if $c$ is not admissible (all terms are $0$), defining $W(\Gamma,\psi)$ as in Equation \eqref{serie-integrale}, the following holds:
\begin{theorem}
Let $(\Gamma,\psi)$ be a trivalent graph equipped with a connection with values in SL$_2(\C)$, and let $R_Y=\C[[Y_e,e\in E]]$. Then the following equality holds in $R_Y$:
\begin{eqnarray*}
W(\Gamma,\psi)&=&\int_{G^{V}}\frac{\dd g}{\prod\limits_{e:(v,w)}\det(\id-\psi_{e,v}g_v\psi_{e,v}^{-1}\psi_{e,w}g_w^{-1}\psi_{e,w}^{-1}Y_e)}
\end{eqnarray*}

\end{theorem}
\subsection{Comparison between generating series}
Let $\Gamma$ be a trivalent graph.
We denote by $\ba{R}_X$ be the sub-algebra of $R_X$ linearly generated by $X^c$ for admissible colorings $c$ and by $\ba{R}_Y$ the sub-algebra of $R_Y$ linearly generated by $Y^c$ for admissible colorings $c$.
We define a morphism (anti-linear on the left) $(\cdot,\cdot):\ba{R}_X\times \ba{R}_X\to \ba{R}_Y$ in the following way: if $c$ and $c'$ are distinct then $(X^c,X^{c'})=0$. If not, we set 
$$(X^c,X^c)=\frac{Y^c}{\prod\limits_{v:(e_1,e_2,e_3)} \la\la \Theta,c_{e_1},c_{e_2},c_{e_3}\ra\ra}$$
The equation \eqref{relation} relating $\la \Gamma,c\ra$ and $[\Gamma,c]$ generalizes directly to the case with holonomies. Summing over all colorings we find:

\begin{equation}\label{bizarre}
(Z(\Gamma,\psi),Z(\Gamma,\psi))=W(\Gamma,\psi).
\end{equation}
It would be interesting to interpret this formula at the level of generating series for which we should interpret the product $(\cdot,\cdot)$ as an hermitian product over $\ba{R}_X$ given by some integral formula. We do not address this question here.
%


%

We end this section with some remarks on the orthogonality of spin-network as functions of the holonomy.
Let $\boM$ be the moduli space of holonomies on $\Gamma$ with values in $G=\textrm{SU}_2$. It is equal to the quotient of $G^H/G^{V\amalg E}$: the Haar measure on $G^H$ produces a measure $\mu$ on $\boM$.

The map from $\boM$ to $\C$ sending $\psi$ to $\la \Gamma,c,\psi\ra$ is a polynomial function and it is well-known that the family $\la \Gamma,c,\cdot \ra$ form an orthogonal basis of $L^2(\boM,\mu)$.
The equation \eqref{relation} relating $\la \Gamma,c\ra$ and $[\Gamma,c]$ generalizes directly to the case with holonomies. We can use it to compute the normalization of the basis:

$$\int_{\boM}|\la \Gamma,c,\psi\ra|^2\dd\mu(\psi)=\prod_v\la v\ra\prod_e\la e\ra\int_{\boM}\prod_{h=(v,e)}\tr_{c_e}(g_e\psi_h g_v\psi_h^{-1})\dd\mu(h)
$$
We have $\int_{G}\tr_c(AgBg^{-1})\dd g= \frac{1}{c+1}\tr_c(A)\tr_c(B)$ by an argument similar to Lemma \ref{prodtrace}. If we apply this formula to each term of the product, the term $\prod_e \la e\ra$ gets inverted. By the orthonormality of characters and Clebsh-Gordan rule, we find $\int_{\boM}|\la \Gamma,c,\psi\ra|^2\dd\mu(\psi)=\prod_v\la v\ra \prod_e\la e\ra ^{-1}$.

\section{Asymptotics of spin-network evaluations}
In this section, we compute the first order of the asymptotic behavior of $[\Gamma,kc,1]$ (the square of the spin-network evaluation divided by $\prod_{v}\la v\ra$), for general $\Gamma$. In Subsection \ref{Subsection:critical}, we transform the integral formula to adapt it to stationary phase approximation and describe the critical set of the integrand.
Then we discuss some sufficient non-degeneracy hypotheses (Subsection \ref{hypo}). The computation of the Hessian occupies Subsection \ref{Subsection:hessian}, in Subsection \ref{Subsection:stationaryphase} we apply the stationary phase method whereas Subsection \ref{Subsection:proof} ends the proof.

\subsection{Description of the critical set}\label{Subsection:critical}
From now on, to simplify the notation, we will write for any $v,w\in V: h(v,w)=\la v,w\ra$. Denote by $S^3$ the unit sphere of $V$ and $\dd v$ the Haar measure on it (i.e. such that $\int_{S^3}1dv=1$).
\begin{lemma}\label{coh}
For any $g\in G$, $\tr_n(g)=(n+1)\int_{S^3} \la v,\rho(g)v\ra^n\dd v$.
\end{lemma}
\begin{proof}
Recall that $\bigoplus_n V_n$ is the algebra of polynomials on $V^*$  and that $G$ acts on it by algebra morphisms. For any $v\in S^3$ , we have $\la v^n,w^n\ra=\la v,w\ra^n$. Consider the endomorphism $U$ of $V_n$ given by $U=\int_{S^3}v^n \la v^n,\cdot\ra\dd v$. This endomorphism clearly commutes with the action of $G$, hence it is proportional to $\id_{V_n}$. As $\tr(v^n \la v^n,\cdot\ra)=1$, we find $\tr(U)=1$ hence $U=\frac{1}{n+1}\id_{V_n}$. Now, given any $g\in G$, we have $\tr_n(g)=\tr \rho_n(g)=(n+1)\tr \rho_n(g)U=(n+1)\int_{S^3}\la v^n,\rho(g)v^n\ra\dd v=(n+1)\int_{S^3}\la v,\rho(g)v\ra^n\dd v$.
 \end{proof}

We deduce the following formula:
$$[\Gamma,c]=\prod_{e}\la e\ra\int_{G^V}\int_{(S^3)^E}\prod_{e:v\to w}\la g_v u_e,g_w u_e\ra^{c_e}\dd g \dd u.$$
 The notation $e:v\to w$ means that $e$ is an oriented edge joining $v$ to $w$, $u_e$ is an element of $S^3$ corresponding to the edge $e$.
 
 In order to analyze this integral, we restrict the domain of the integral such that the integrand does not vanish and such that symmetries behave nicely: we will see that this restriction does not affect the asymptotic analysis. More precisely,  let $X$ be the following subset of $G^V\times (S^3)^E$ and 
 $\pi:S^3\to S^2$ be the Hopf fibration. The sphere $S^2$ is either identified to the projective space $\mathbb{P}(V)$ or to the unit sphere of the Lie algebra $\bog$ of $G$ (equipped with the scalar product $|\xi|^2=-\frac{1}{2}\tr(\xi^2)$ ).
 \begin{eqnarray*}
 X&=&\{(g_v,u_e)\in G^V\times (S^3)^E\textrm{ such that }\forall e:v\to w,\la g_v u_e,g_w u_e\ra\ne 0\\
 &&\textrm{ and the family } (\pi(u_e))_{e\in E}\textrm{ has rank at least 2 in }\bog\}
\end{eqnarray*}

Let $F:X\to \C /2i\pi \Z$ be the map $F(g,u)=\sum\limits_{e:v\to w} c_e \ln \la g_v u_e,g_w u_e\ra$ and write $P_e=\pi(u_e)\in \bog$ if the orientation of $e$ coincides with the chosen one and $P_e=- \pi(u_e)$ if not.
 We have the following proposition:
 \begin{proposition}\label{critical}
The critical points of $F$ are the elements $(g,u)\in X$ satisfying 
\begin{enumerate}
\item
For all edges $e:v\to w$,  $g_v u_e=\tau_e g_w u_e$ for some $\tau_e\in S^1$.
\item
For all vertices $v$ with edges $e_1,e_2,e_3$ going out of $v$, one has $c_{e_1}P_{e_1}+c_{e_2}P_{e_2}+c_{e_3}P_{e_3}=0$.
\end{enumerate}
\end{proposition}
\begin{proof}

Derivate $F$ with respect to the variable $u_e$ in the direction $v_e$. We have $DF(v_e)=\frac{c_e}{\la g_v u_e,g_w u_e\ra}(\la g_v v_e,g_w u_e\ra+\la g_v u_e,g_w v_e\ra)$.
Denote by $u_e^{\perp}$ the element of $S^3$ satisfying $\la u_e,u_e^{\perp}\ra=0$ and $\det(u_e,u_e^{\perp})=1$. Then, write $g_v^{-1}g_{w}$ in the basis $(u_e,u_e^{\perp})$. Taking $v_e=u_e^{\perp}$ and $v_e=i u_e^{\perp}$, we find that $g_v^{-1}g_w$ is diagonal in the basis $(u_e,u_e^{\perp})$. This is equivalent to the first item of the proposition.

Consider an element $\xi$ in Lie$(G)$ and replace $g_v$ by $g_v e^{t\xi}$ for some $t>0$. Then compute the derivative of $F$ at $t=0$. Three terms contribute to the sum, corresponding to the edges incoming to $v$. Suppose that they are oriented from $v$ to $v_1,v_2,v_3$ by edges denoted by $e_1,e_2,e_3$ respectively. For simplicity, we note $g=g_v$, $g_i=g_{v_i}$, $u_i=u_{e_i}$ and $c_i=c_{e_i}$.

We have $F'(0)=\sum_{i=1}^3 c_i \frac{\la g \xi u_i,g_i u_i\ra}{\la g u_i,g_i u_i\ra}=\sum_{i=1}^3 c_i \la \xi u_i,u_i\ra$. The last equality comes from the first item. 
Let $\Pi_i$ be the element of End$(V)$ acting by $i$ on $u_i$ and $-i$ on $u_i^{\perp}$. 
Then, from $\tr(\xi \Pi_i)=2i\la u_i,\xi u_i\ra$ we have that $\sum_i c_i \Pi_i$ is $0$. Using the standard identification between anti-hermitian operators and $\bog$ which sends $\Pi_i$ to $P_i$, we obtain the second item of the proposition.
\end{proof}
Suppose that the coloring $c$ satisfy the strict triangle inequalities (T): then at any vertex, the triple $P_{e_1},P_{e_2},P_{e_3}$ has rank 2. In particular, the corresponding pair $(g_v,u_e)$ is in $X$. We will make this assumption in the sequel.

The integral presents some obvious symmetries: the integrand depends on $u_e$ through its image $\pi(u_e)\in S^2$. Hence, we can replace the integral over $(S^3)^E$ by an integral over $(S^2)^E$ integrating over the fiber of the Hopf fibration $(S^1)^E$. Moreover replacing $(g_v,u_e)$ by $(gg_v,u_e)$ do not change the integral as replacing $(g_v,u_e)$ by $(g_vg^{-1},g u_e)$. Hence, the integral can be performed over $Y=X/(S^1)^E\times G\times G$ where $(\alpha_e,g,h)$ acts on $(g_v,u_e)$ by $(gg_vh^{-1},h(\alpha_e u_e))$. Notice that the stabilizer of the action of $(S^1)^E\times G\times G$ on $X$ is $\{\pm 1\}$, hence the quotient $Y$ is a smooth manifold of dimension $12N-6$. Let us denote by $\tilde{F}:Y\to \C/2i\pi\Z$ the induced map.

We denote by $I\subset \bog^E/G$ the isometry classes of tuples $(P_e)_{e\in E}$ satisfying the equation (2) of Proposition \ref{critical}. They encode the critical points of $\tilde{F}$ in the following way: 
\begin{proposition}\label{critical2}
The set $C$ of critical points of $\tilde{F}$ are in bijection with the set of triples $(P,Q,(g_v))$ in $I\times I\times \Big(G^V/\{\pm 1\}\Big)$ where for all half-edges $(e,v)$ we have $g_v P_e=Q_e$. Moreover, the map $C\to I\times I$ which forgets the third term is surjective and its fibers have cardinality $2^{2N-1}$. 
\end{proposition}

We remark that if $N>1$, $I$ is never reduced to 1 element as if $P$ belongs to $I$, $-P$ is an element of $I$ distinct from $P$.
\begin{proof}
Let $(g,u)$ be a critical point of $F$. Then the family $P_e=\pi(u_e)$ belongs to $I$. 
For any half-edge $(e,v)$, set $Q_e=\pi(g_v u_e)=\pi (g_w u_e)$. For any vertex $v$ with edges $e_1,e_2,e_3$ incoming to it, the relation $\sum_i c_{e_i}P_{e_i}=0$ implies $\sum_i c_{e_i}Q_{e_i}=0$, hence family $(Q_e)$ also belongs to $I$. 
Conversely, given two families $(P_e),(Q_e)$ in $I$, we choose vectors $u_e,s_e\in S^3$ satisfying $\pi(u_e)=P_e$ and $\pi(s_e)=Q_e$ . Then for every vertex $v$ surrounded by $e_1,e_2,e_3$, we need to find  an element $g_v\in G$ such that for all $i$, $g_v u_{e_i}$ is proportional to $s_{e_i}$. As we required that the coefficients $c_{e_i}$ satisfy the strict triangular inequalities, the pair $P_{e_1},P_{e_2}$ is linearly independent in $\bog$. As $Q_{e_1},Q_{e_2}$ have same lengths and same angle by hypothesis, there is a unique rotation of $\bog$ which sends $P_{e_i}$ to $Q_{e_i}$ for all $i$. This rotation lifts in two possible ways in $G$ giving our choices for $g_v$. Note that the action of $\{\pm 1\}\in G\times G$ divides by 2 the number of critical points.
\end{proof}

\subsection{Non-degeneracy hypothesis}\label{hypo}

We would like to understand which conditions ensure that the critical points of $\tilde{F}$ are isolated in $Y$. At first, we concentrate on the family of vectors $(P_e)$ associated to critical points.

Let $\boE$ be the set of maps from $E$ to $\bog$. We consider it as an Euclidian space where the scalar product is induced from the scalar product $|\xi|^2=-\frac{1}{2}\tr(\xi^2)$ on $\bog$. This space will then be identified alternatively as the chain or cochain complex of $\Gamma$ with coefficients in $\bog$. As edges of $E$ are oriented, there is a boundary operator $\partial:\boE\to C_0(\Gamma,\bog)$ and we set $\boH=\ker\partial$. We see that any critical point in $X$ produces an element $\zeta$ of $\boH$ by setting $\zeta_e=c_eP_e$. Conversely, any element $\zeta$ of $\boH$ with $|\zeta_e|=c_e$ produces an element of $I$ (see Proposition \ref{critical2}). 

A tangent vector to $\zeta$ is an element $\zeta'\in \boH$ such that $\la \zeta_e,\zeta'_e\ra=0$ for all $e$ (because $|\zeta_e|=c_e$). We interpret this formula as the scalar product of $(\zeta_e \delta_e)$ and $\zeta'$. On the other hand, a tangent vector correspond to a global symmetry if and only if there exists $\xi\in\bog$ such that $\zeta'_e=\xi\times \zeta_e$ for all $e$. 
We count that the dimension of $\boH$ is $3(N+1)$, the number of conditions imposed by $\zeta$ is $3N$ (one for each edge), whereas the symmetry gives 3 dimensions (because the family $\zeta_e$ has rank at least 2). Hence, an element of $\bog^E/G$ associated to $P$ is isolated if the family $(P_e\delta_e)$ projected orthogonally on $\boH$ is linearly independent. With this interpretation in mind, we make the following assumption:

\begin{hypothesis}\label{H1} Let $(c_e)$ be a coloring of $\Gamma$ such that all triangle inequalities in (T) are strict. We suppose that all elements $\zeta\in \boH$ such that $|\zeta_e|=c_e$ satisfy the non-degeneracy condition that the family $(\Pi_\boH(\zeta_e\delta_e))_{e\in E}$ is linearly independent where $\Pi_\boH:\boE\to \boH$ is the orthogonal projector on $\boH$. This implies that $I$ is a finite set.
\end{hypothesis}

\begin{remark}\label{hypo-remark}
The preceding hypothesis is equivalent to saying that for any $(\xi_v)\in C_0(\Gamma,\bog)$, if $P_e\times(\xi_w-\xi_v)=0$ for all edges $e:v\to w$ then all the $\xi_v$ are equal. Let us show it:

The image of $d:C_0(\Gamma,\bog)\to\boE$ is orthogonal to the space $\boH$. Let $\boU$ be the subspace $\boU=\oplus_e \R P_e \subset \boE$. The non-degeneracy assumption is that the orthogonal projection of $\boU$ on $\boH$ is injective. Hence, by standard linear algebra, the orthogonal projection of $\boU^{\perp}$ on $\boH^{\perp}$ is surjective. This implies that if some vector $\zeta_e$ in $\boH^{\perp}$ satisfies $P_e\times \zeta_e=0$ for all $e$ then $\zeta_e=0$. 
\end{remark}

For any critical point $x$, one can define the \emph{phase function} $\tau^x:E\to S^1$ in the following way. Set $x=(g_v,u_e)$. Then, by the first item of Proposition \ref{critical}, $g_vu_e$ and $g_w u_e$ are proportional unit vectors, moreover the phase factor $\tau^x_e=\la g_v u_e,g_w u_e\ra$ depends only on the class of the critical point $x$ in $Y$. We remark that the generalized phase of a critical point indexed by $(P,Q,g_v)\in I\times I\times G^V$ depends on $g_v$ only up to a sign and is necessarily equal to $\pm 1$ if $P=Q$. We assume that this is the only case when the phase function takes the value $\pm 1$:

\begin{hypothesis}\label{H2} For any $P$,$Q$, distinct elements of $I$, the phase function of the associated critical points does not take the value $\pm 1$.
\end{hypothesis}

Let $P,Q$ be distinct elements of $I$, and $\tau$ be the associated phase function. Write $\tau_e=e^{i\theta_e}$. Then, $\theta_e$ is well-defined modulo $\pi$. Consider the following quadratic form on $\bog^V$:
$$q=\sum_e 2c_e\left( -i\cot(\theta_e)||Q_e\times(\xi_v-\xi_w)||^2+i\la Q_e,\xi_v\times\xi_w\ra\right).$$

\begin{hypothesis}\label{H3}
 For any $P$,$Q$ distinct elements of $I$, the quadratic form $q$ has co-rank 6.
\end{hypothesis}

Let us develop an example which is the main motivation of this section:
\begin{example}
Suppose that $\Gamma$ is a planar graph and that $I=\{P,-P\}$. Then there is up to isometry a unique polyhedron $\Delta\subset \bog$ whose 1-skeleton is dual to $\Gamma$ (hence whose faces are triangles) and such that for any oriented edge $e$ of $\Gamma$ the dual edge in $\Delta$ is vectorially equal to $c_eP_e$: in particular, it has length $c_e$. 
\end{example}
The non-degeneracy condition $(H1)$ is equivalent to an infinitesimal rigidity condition on $\Delta$. For instance, if $\Delta$ is convex, this condition is automatically satisfied by Cauchy's theorem (see for instance \cite{AiZi}, Chapter 13). The generalized phase function has the following nice interpretation: let $(g_v,u_e)$ be a critical point associated to the pair $(P,-P)$. This means that $\pi(u_e)=P_e$ and $\pi(g_vu_e)=-P_e$. Let $v$ be a vertex of $\Gamma$, dual to a face $F_v$ of $\Delta$. Suppose that $v$ has three outgoing edges $e_1,e_2,e_3$. Then, $g_v$ lifts the unique rotation mapping $P_{e_i}$ to $-P_{e_i}$, that is the rotation of $\pi$ in the plane supporting $F_v$. Now, given an edge $e$, one has $\tau_e=\la g_v u_e,g_w u_e\ra=\pm e^{i\theta}$ where $2\theta$ is the angle of the rotation $g_v^{-1}g_w$ around $P_e$. Hence, $\theta$ is the angle modulo $\pi$ between the faces $F_v$ and $F_w$. 
In particular, if $\Delta$ is a non-degenerate polyhedron, its corresponding graph satisfies (H2). We do not know which geometric condition on $\Delta$ would imply that (H3) is valid.

\subsection{Second variation formula}\label{Subsection:hessian}
Given a point $(g_v,u_e)\in X$, we can build a a coordinate system around it using the variables $(\xi_v)\in\bog^V$, $(\lambda_e)\in \C^E$ and $\alpha_e\in (S^1)^{E}$. The parametrization is given by $(e^{\xi_v}g_v,\alpha_e u_e(\lambda_e))$ where for any $u\in S^3$ and $\lambda\in \C$ we set 
$$u(\lambda)=\frac{1}{\sqrt{1+|\lambda|^2}}(u+\lambda u^{\perp})=(1-|\lambda|^2/2)u+\lambda u^{\perp}+o(|\lambda|^2).$$ 
Let $s^e_v=g_v u_e$ and notice that $g_v u_e(\lambda_e)=s^e_v(\lambda_e)$.
Let $ix_v^e=\la s^e_v,\xi_v s^e_v\ra$, $\mu_v^e=\la s_v^{e\perp},\xi_v s^e_v\ra$, $\tau_e=\la s^e_v,s^e_w\ra$. In what follows, the values of $g_v,u_e,s^e_v$ are to be considered as fixed and $\xi_v,\lambda_e$ (and consequently those of $x^e_v,\mu^e_v$) as variable. Notice that $x^e_v\in \mathbb{R}$, that $\tau_e$ was already defined in Proposition \ref{critical}, and that $F:X\to R$ does not depend on the coordinates $\{\alpha_e\}$, so we set them to $1$ in the following:
\begin{proposition}\label{order2development}
The Taylor expansion of $F$ in a neighborhood of a point $(g_v,u_e)\in X$ is up to the second order: 
\begin{eqnarray*}
F(e^{\xi_v}g_v,u_e(\lambda_e))&=&\sum_e c_e\ln \tau_e+\sum_e c_e i(x^e_w-x^e_v)\\
&&+\sum_e c_e \frac{\overline{\tau_e}}{\tau_e}(\ba{\mu_v^e}+\ba{\lambda_e})(\mu_w^e+\lambda_e)\\
&-&\sum_e c_e\big(|\lambda_e|^2+|\mu_v^e|^2/2+|\mu_w^e|^2/2+\ba{\lambda_e}\mu_v^e+\lambda_e\ba{\mu_w^e}\big)\\
&+&o(\sum_v|\xi_v|^2+\sum_e|\lambda_e|^2)
\end{eqnarray*}
In particular, on critical points $\sum_e i(x^e_v-x^e_w)=0$ and $\frac{\overline{\tau_e}}{\tau_e}=\tau_e^{-2}$.
\end{proposition}
 \begin{proof}
Notice that in the basis $(s^e_v,(s^e_v)^{\perp})$ we have $\xi_v=\begin{pmatrix} ix_v^e & -\ba{\mu_v^e}\\ \mu_v^e & -i x^e_v\end{pmatrix}$; moreover since $\la s^e_v,s^e_w\ra=\tau_e$, then $\la(s^e_v)^{\perp},(s^e_w)^{\perp}\ra=\overline{\tau_e}$ ( $=\tau_e^{-1}$ on critical points by Proposition \ref{critical}). 
To compute the first terms of $F(e^{\xi_v}g_v,u_e(\lambda_e))=\sum\limits_{e:v\to w} c_e \ln \la e^{\xi_v}s^e_v(\lambda_e),e^{\xi_w}s^e_w(\lambda_e)\ra$, we first compute 
$$e^{\xi_v}s^e_v(\lambda_e)=\begin{pmatrix} 1+ix_v^e-|\xi_v|^2/2& -\ba{\mu_v^e}\\ \mu_v^e& 1-ix_v^e-|\xi_v|^2/2\end{pmatrix}\begin{pmatrix} 1-\frac{|\lambda_e|^2}{2} \\\lambda_e\end{pmatrix}+h.o.t.
$$
Then we get $\la e^{\xi_v}s^e_v(\lambda_e),e^{\xi_w}s^e_w(\lambda_e)\ra=\tau_e(1-ix_v^e+ix_w^e+x_v^ex_w^e-|\lambda_e|^2-|\xi_v|^2/2-|\xi_w|^2/2-\mu^e_v\ba{\lambda_e}-\ba{\mu^e_w}\lambda_e)+\overline{\tau}_e(\ba{\mu^e_v}+\ba{\lambda_e})(\mu^e_w+\lambda_e)+h.o.t.$ Taking the logarithm, expanding it up to order $2$ terms, and recalling that $|\xi_v|^2=-\frac{1}{2}\tr (\xi_v^2)=(x^e_v)^2+|\mu^e_v|^2$ we get the formula of the proposition. The last statement is a consequence of Proposition \ref{critical}.
 \end{proof}
 
 To obtain a nicer formula, in the case of a critical point, we introduce the variables $z_v^e=\mu_v^e$ and $z_w^e=\tau_e^{-2}\mu_w^e$. 
After some computation, by Proposition \ref{order2development}, one has that on critical points ${\textrm Hess}(F)_x(\xi,\lambda)=-q(\xi,\lambda)$ where: 
\begin{eqnarray}\label{hessien}
q(\xi,\lambda)&=&\sum_e c_e\big(2(1-\tau_e^{-2})|\lambda_e|^2+2\lambda_e\tau_e^{-2}(\ba{z_w^e}-\ba{z_v^e})+2\ba{\lambda_e}(z_v^e-z_w^e)\notag\\&&+|z_v-z_w|^2+z_v\ba{z_w}-z_w\ba{z_v}\big)\\ 
&=&\sum_e c_e \Big(2(1-\tau_e^{-2})\Big|\lambda_e-\frac{z^e_w-z^e_v}{1-\tau_e^{-2}}\Big|^2\notag\\
&& +\frac{\tau_e+\tau_e^{-1}}{\tau_e-\tau_e^{-1}} |z^e_w-z^e_v|^2+z^e_v\ba{z^e_w}-z^e_w\ba{z^e_v}\Big)\text{ if }\tau_e^2\ne 1\forall e\in E.
\end{eqnarray}

\subsection{Applying the stationary phase method}\label{Subsection:stationaryphase}
Let us now perform the stationary phase approximation when replacing the coloring $c$ by $kc$ and letting $k$ go to infinity in $\int_{Y} \exp(k\tilde{F})\dd \mu$. In this formula, the measure $\mu$ is obtained from the Haar measure on $(S^3)^{2N}\times (S^3)^{3N}$ by integration over the action of $(S^1)^E\times G\times G/\{\pm 1\}$, equipped with its Haar density. 
In this subsection as well as in the following we rely on the notation and machinery recalled in Appendix \ref{gaussian}.

By assumption, any critical point $x$ is isolated in $Y$. Provided that ${\textrm Hess}(\tilde{F})_x$ is non-degenerate, we can apply the stationary phase expansion theorem (Theorem 7.7.5 of \cite{H}) to $\tilde{F}$ which is a smooth function with non-positive real part, and obtain that the local contribution of $x$ to $[\Gamma,kc]$ when $k$ goes to infinity is: 

\begin{equation}\label{phasestationnaire}
I(x)=\prod_e \la e\ra\  e^{k\tilde{F}(x)}\left(\frac{(2\pi)^{12N-6}}{\det(-k {\textrm Hess}(\tilde{F})_x,\mu)}\right)^{1/2}=\prod_e \la e\ra\  e^{k\tilde{F}(x)}I(kq).
\end{equation}

The proof of Theorem \ref{asymptotic} consists then in computing for any critical point the gaussian integral $I(q)$ where $q$ is the opposite of the Hessian computed in the preceding subsection. We observe that in Equation \eqref{hessien} this quadratic form is expressed in terms of $\{ \lambda_e,z^e_v,\overline{z}^{e}_v\}$ and that its restriction to $\{z^e_v=z^e_w=0\}$ is diagonal. Thus we would like to perform a partial integration on the $\lambda_e$ terms which form local coordinates for the $(S^2)^E$-part of $Y$. 

There are two different cases to handle: in the first case one has $\tau_e^2=1$ for all $e$ in $E$. The quadratic form $q$ is then degenerate with respect to the $\lambda$ coordinates and we deform it with a parameter $\kappa$. In the second case, one has $\tau_e^2\ne 1$ for all $e$, thanks to Hypothesis 2. Fix a critical point described by a pair $(P,Q)\in I\times I$ (as in Proposition \ref{critical2}); to avoid a cumbersome notation, in what follows we will suppress the indices $P,Q$ except in the definition of the quadratic forms used in the statement of Theorem \ref{asymptotic}.
\subsubsection{Deformation and degeneracies}
The kernel of $q$ contains the image of $\bog\times \bog$. Let $\xi$ and $\eta$ be the two variables acting on $(g_v,u_e)$ by $(e^{\xi}g_v e^{-\eta},e^{\eta}u_u)$. For an edge $e:v\to w$ set $p_e=\la u_e^{\perp},\eta u_e\ra$ and $q_e=\la s_v^{e\perp},\xi s^e_v\ra$. The infinitesimal action of $\xi$ and $\eta$ in the tangent space at $(g_v,u_e)$ is $(\xi-g_v\eta g_v^{-1},\eta)$. One computes in terms of $p_e$ and $q_e$ this action as $\lambda_e=p_e$, $z_v^e=\mu_v^e=\la s_v^{e\perp},(\xi-g_v\eta g_v^{-1})s^e_v\ra=q_e-p_e$ and $z_w^e=q_e-\tau_e^{-2}p_e$. 

Consider the deformation $q_\kappa$ of $q$ obtained by replacing $\tau_e$ by $\tau_{e\kappa}=\kappa \tau_e$. One recovers the original one by letting $\kappa\in (1,+\infty)$ go to 1. After some computations one may check that we have $q_{\kappa}|_{\bog\times\bog}=4(\kappa-1)\sum_e c_e |p_e|^2+o(\kappa-1)$ 
Hence, denoting by $r_P$ the quadratic form on $\bog$ given by $r_P(\eta)=\sum_e c_e |p_e|^2$ we can remove the indeterminacy in $\eta$ as explained in Appendix \ref{gaussian}. We have: $$I(q)=\lim_{\kappa\to 1}\frac{I(q_\kappa)}{I(q_\kappa|_{\bog\times\bog})}=\lim_{\kappa\to 1}\frac{(4\kappa-4)^{3/2}I(q_\kappa)}{I(r_P)}.$$

\subsubsection{Partial integration}
We now remark that this deformation also allows to integrate partially over the coordinates $(\lambda_e)$. More precisely, let $\sigma_{\kappa}$ be the restriction of $q_\kappa$ to the $\lambda$-coordinates, setting $\sigma_{\kappa}=\sum_e 2c_e(1-\tau_{e\kappa}^{-2})|\lambda_e|^2$, 
one obtains $I(q_\kappa)=I(\sigma_{\kappa})I(q^{\kappa})$ where 
$$q^{\kappa}=\sum_e c_e\left( \frac{\tau_{e\kappa}+\tau_{e\kappa}^{-1}}{\tau_{e\kappa}-\tau_{e\kappa}^{-1}}|z_v^e-z_w^e|^2+z_v^e\ba{z_w^e}-\ba{z_v^e}z_w^e\right).$$

We can reformulate the preceding quadratic form by introducing only operations on $\bog$.
For instance, one has $|z_v^e|^2=|Q_e\times \xi_v|^2$. We also have $z_v^e\ba{z_w^e}-z_w^e\ba{z_v^e}=-i\tr(\Pi\xi_v\xi_w)=2i\la Q_e,\xi_v\times \xi_w\ra$, where $\Pi$ is the matrix whose eigenvalues in the basis $(s^e_v,s_v^{e\perp})$ are $i$ and $-i$ and $2\xi_v\times\xi_w=[\xi_v,\xi_w]$. This may be proved by direct computation in the basis $(s^e_v,s_v^{e\perp})$. Hence, we can write:

\begin{equation}\label{eqqk'}
q^{\kappa}_{P,Q}=\sum_e c_e\left( \frac{\tau_{e\kappa}+\tau_{e\kappa}^{-1}}{\tau_{e\kappa}-\tau_{e\kappa}^{-1}}|Q_e\times(\xi_v-\xi_w)|^2+2i\la Q_e,\xi_v\times\xi_w\ra\right).
\end{equation}

{\bf First case:} $\tau_e^2=1$ for all $e$ in $E$. \\
Write $q^{\kappa}=\frac{\kappa^2+1}{\kappa^2-1}q_Q+q'$ where we set $q'=\sum_e 2i c_e\la Q_e,\xi_v\times\xi_w\ra$ and $$q_Q=\sum_e c_e|Q_e\times(\xi_v-\xi_w)|^2.$$

Using remark \ref{hypo-remark}, the kernel of $q_Q$ is the subspace $\bog$ given by the equations $\xi_v=\xi$ for all $v$. When $\kappa$ goes to $1$, the quadratic form $q'$ becomes negligible.
Hence, we have $I(q^{\kappa})\simeq(\kappa-1)^{3N-3/2}I(q_Q)$, and using $I(|\lambda|^2,\mu_H^{\P^1})=I(|\lambda|^2,\mu_{euc}^{\P^1})/\pi=2$ one computes:
$$I=\lim_{\kappa\to 1}\frac{(4\kappa-4)^{3/2}}{I(r_P)}\frac{2^{3N}}{(4\kappa-4)^{3N}\prod_e c_e}
\frac{I(q_Q)}{(\kappa-1)^{3/2-3N}}=\frac{2^{3-3N}I(q_Q)}{I(r_P)\prod_e c_e}.$$

{\bf Second case:} $\tau_e^2\ne 1$ for all $e$ in $E$.

Setting $\kappa=1$ in $q^{\kappa}$, we get the quadratic form $q''$: 
\begin{equation}\label{eqq'}
q''_{P,Q}=\sum_e c_e\left( -i \cot(\theta_e)|\Pi_{Q_e^{\perp}}(\xi_v-\xi_w)|^2+2i\la Q_e,\xi_v\times\xi_w\ra\right).
\end{equation}

By hypothesis (H3) this form has corank 6 hence, the following formula makes sense:
\begin{eqnarray}
I&=&\lim_{\kappa\to 1}\frac{(4\kappa-4)^{3/2}}{I(r_P)}I(\sigma_1)I(q^{\kappa})\notag
=\lim_{\kappa\to 1}\frac{(4\kappa-4)^{3/2}}{I(r_P)}\frac{2^{3N}}{\prod_e 2c_e(1-\tau_e^{-2})}I(q^{\kappa})\\
&=&\lim_{\kappa\to 1}\frac{2^{3}(\kappa-1)^{3/2}I(q^{\kappa})}{I(r_P)\prod_e c_e(1-\tau_e^{-2})}=\lim_{\kappa\to 1}\frac{2^{3-3N}i^N\prod_e\tau_e(\kappa-1)^{3/2}I(q^{\kappa})}{\prod_e c_e\sin(\theta_e)I(r_P)}
\label{contributionPQ}
\end{eqnarray}

\subsubsection{Collecting the critical points}\label{Subsection:proof}
In this section, we apply the previous computations to all critical points and collect the results. 
There are two cases: whether the critical point corresponds to a pair $(P,P)$ or to a pair $(P,Q)$ with $P\ne Q$. 

{\bf First case:}
Fix a $3N$-tuple $(u_e)$ representing $P_e$. All critical points associated to the pair $(P,P)$ have the form $((-1)^{\epsilon_v},u_e)$ where $\epsilon_v\in\{\pm 1\}$. 
In the preceding subsection, we obtained that the contribution of a single critical point $(1,u_e)$ to \eqref{phasestationnaire} is $2^{3-3N}I(q_P)/I(r_P)$, remarking that the term $\prod_e c_e$ cancels and using the fact that $\tilde{F}(1,u_e)=0$. The other pairs $((-1)^{\epsilon_v},u_e)$ differ only by the value of $e^{k\tilde{F}}=\prod\limits_{e:v\to w} (-1)^{(\epsilon_v+\epsilon_w)kc_e}$. Adding the contributions and dividing by 2 (see Proposition \ref{critical2}),  we get:
$$\frac{1}{2}\Big(\sum_{\epsilon_v}\prod_{v:(e_1,e_2,e_3)}(-1)^{\epsilon_vk(c_{e_1}+c_{e_2}+c_{e_3})}\Big)\frac{2^{3-3N}I(q_P)}{I(r_P)}=\frac{2^{2-N}I(q_P)}{I(r_P)}.$$

We note that the contribution of $(-P,-P)$ is the same because $q_{-P}=q_P$ and $r_{-P}=r_P$. We will multiply the result by 2 when summing over $P\in I/\{\pm 1\}$.

{\bf Second case:}
In the case when $P\ne Q$, we obtained in Equation \eqref{contributionPQ} that the contribution of a critical point $(g_v,u_e)$ is 
$$\lim_{\kappa\to 1}\frac{2^{3-3N}i^N(\kappa-1)^{3/2}I(q^{\kappa}_{P,Q})}{I(r_P)\prod_e \sin(\theta_e)}e^{i\sum_e (kc_e+1)\theta_e}.$$
 As before, taking into account all critical points associated to the same pair $(P,Q)$ amounts in multiplying the result by $2^{2N-1}$. 
 
 When replacing $(P,Q)$ by $(-P,-Q)$, we also have $r_{-P}=r_P$. Given a critical point $(g_v,u_e)$ corresponding to $(P,Q)$, one checks directly that the point $(g_v,u_e^{\perp})$ corresponds to $(-P,-Q)$ and that the phase function gets inverted.
 Letting $q^{\kappa}_{P,Q}$ be defined as in Equation \eqref{eqqk'}, one sees that $q^{\kappa}_{-P,-Q}=-q^{\kappa}_{P,Q}$ and since $q^{\kappa}_{P,Q}$ is purely imaginary when $\kappa\to 1$, it holds: $\lim\limits_{\kappa\to 1}(\kappa-1)^{3/2}I(-q^{\kappa}_{P,Q})=\lim\limits_{\kappa\to 1}(\kappa-1)^{3/2}\overline{I(q^{\kappa}_{P,Q})}$. Collecting the contributions of $(P,Q)$ and $(-P,-Q)$, we get the following formula:
 
$$ \frac{2^{3-N}\textrm{Re}\Big(i^{N}\lim\limits_{\kappa\to 1}(\kappa-1)^{3/2}I(q^{\kappa}_{P,Q})e^{i\sum_e (kc_e+1)\theta_e}\Big)}
{I(r_P)\prod_e \sin(\theta_e)}$$

It remains to compute explicitely the Gaussian integrals involved in the preceding computation. Considering on $\bog\times \bog$ the usual Haar measure, one has to divide the result by 2 because of the isotropy subgroup $\{\pm 1\}\in G\times G$.
Using Appendix \ref{gaussian}, we can replace all occurrences of $I$ by a usual $\det$ in the euclidean basis of $\bog$ and its tensor powers. Denoting by $\det'$ the product of the non-zero eigenvalues of a matrix, we have 
\begin{eqnarray*}
I(r_P)&=&(2/\pi)^{1/2}\det(r_P)^{-1/2} \\
I(q_P)&=& (2N)^{3/2} (2/\pi)^{N-1/2}\det\vphantom{}\!'(q_P)^{-1/2}\\
I(q^\kappa_{P,Q},\mu)&=& (2N)^{3/2}(2/\pi)^{N-1/2}\det\vphantom{}\!'(q^\kappa_{P,Q})^{-1/2}
\end{eqnarray*}
where the factors $(2N)^{3/2}$ are due to the fact that the Haar density on the kernels of $q_P$ and of $q_{P,Q}^{\kappa}$ is $\mu_{Haar}=\frac{\mu_{euc}}{(2N)^{3/2}}$ because each vector of the kernel is represented in ${\cal G}^{2N}$ by $2N$-copies of the same vector of $\cal{G}$. 
\appendix
\section{Super-symmetric rules}
We refer to \cite{dm} for a detailed discussion. We collect here some definitions and warnings important in the article.
A super vector space $V$ is a direct sum of two finite dimensional complex vector spaces $V_0$ and $V_1$ called respectively even and odd part. We will always suppose that our spaces are homogeneous that is one of the two components vanishes, in particular any element $x\in V$ has a degree $|x|$ which is 0 or 1 depending on the parity of $V$. Super vector spaces form a category where morphism are linear maps respecting the decomposition. 

The tensor product $V\otimes W$ is defined as the vector space whose even part is $V_0\otimes W_0\oplus V_1\otimes W_1$ and odd part is $V_1\otimes W_0\oplus V_0\otimes W_1$. There is an isomorphism $c_{V,W}:V\otimes W\to W\otimes V$ sending $x\otimes y$ to $(-1)^{|x||y|}y\otimes x$. The even vector space $\C$ is neutral for tensor product.

This allows to define the unordered tensor product of homogeneous super spaces in the following way:
 Given a finite family of vector spaces $(V_i)_{i\in I}$ of parity $p_i$, we set
$$\bigotimes_{i\in I} V_i=\lim_{\longrightarrow} \,\,(\phi_{(\sigma')^{-1}\circ\sigma}:V_{\sigma}\to V_{\sigma'}).$$

The projective system is defined as follows: let $n$ be the cardinality of $I$, for any bijection $\sigma:\{1,\ldots,n\}\to I$, set $V_{\sigma}=\bigotimes_{i=1}^n V_i$. Then, define an isomorphism $\phi_{(\sigma')^{-1}\circ\sigma}:V_{\sigma}\to V_{\sigma'}$ by the formula 
$$\phi_{\tau}(v_1\otimes\cdots\otimes v_n)=(-1)^sv_{\tau_1}\otimes \cdots \otimes  v_{\tau_n}\text{ where }s=\sum\limits_{i<i',\tau_i>\tau_{i'}} p_i p_{i'}.$$

There is an internal functor $\shom$ satisfying the following adjunction formula:
$$\hom(U,\shom(V,W))=\hom(U\otimes V,W)$$
The space $\shom(U,V)$ is the space of all linear maps where a map is considered even if it respects the parity and odd if it reverses it.
In particular, if we set $U^*=\shom(U,\C)$ we have an isomorphism $\shom(U,V)=V\otimes U^*$ and an evaluation map ev$_U:U^*\otimes U\to \C$. 

With these identifications, bilinear forms are elements of $\hom(U\otimes V,\C)=\hom(U,V^*)=V^*\otimes U^*$. The inversion of the terms should be noticed. 

There is another tricky issue: the natural isomorphism $\theta_U:U\to (U^*)^*$ is not the identity but sends $x$ to $(-1)^{|x|}$. With that convention, we have the identity ev$_{U^*}=$ev$_U\circ c_{U,U^*}\circ (\theta_U^{-1}\otimes 1)$. Unformally, this formula is explained by the transposition of the terms in the Gelfand transform: $\theta(x)(\lambda)=(-1)^{|x||\lambda|}\lambda(x)$.

\section{Gaussian integrals and densities}\label{gaussian}
In this subsection let $F$ be a real vector space of dimension $n$, $F^*$ be its dual, $Q$ a quadratic form on $F$ and $S(F)$ the symmetric algebra of $F$. If $Q$ is non-degenerate, one can identify $F$ and $F^*$ and thus transport it to a quadratic form denoted $Q^{-1}$ on $F^*$.  Note that if $F$ is equipped with a basis and $F^*$ with the dual, the matrices expressing $Q$ and $Q^{-1}$ in these bases are inverse to each other. 

\begin{definition}[Densities]
Let ${\cal B}(F)$ be the set of bases of $F$. A $\emph{density}$ $\mu$ on $F$ is a map $\mu:{\cal B}(F)\to \mathbb{R}$ such that for each $A\in GL(F)$ it holds $\mu(A\cdot b)=|\det(A)|\mu(b)$. The set of densities is a real $1$-dimensional vector space denoted $|\Lambda|(F)$.
A \emph{density} on a manifold $M$ is a continuous section of the bundle $|\Lambda|(TM)$. 
\end{definition}
  
Fixing $b\in {\cal B}(F)$, defines an isomorphism $F\simeq\mathbb{R}^n$ and thus $F$ can be equipped with a density which we denote $\mu_{euc}$ which satisfies $\mu_{euc}(b)=1$.
If $F$ is equipped with a density $\mu$ and $Q:F\to \mathbb{C}$ is a quadratic form, then one defines $\det(Q,\mu)\in \mathbb{C}$ as $\det(Q(b))/\mu(b)^{2}$ where $Q(b)$ is the matrix expressing $Q$ in the basis $b$ of $F$ (this clearly does not depend on $b$).
\begin{lemma}[Gaussian Integrals]
Given a density $\mu$ on $F$ and a quadratic form $Q$ such that $Re(Q)>0$, then \begin{equation}\label{gauss}
I(Q)=\int_F \exp(-\frac{Q}{2})d\mu=\frac{(2\pi)^{n/2}}{\sqrt{\det(Q,\mu)}}\end{equation} 
where the square root is the analytical extension of the positive one on the set of real and positive quadratic forms.
\end{lemma} 
An element $P$ of $S(V^*)$ may be interpreted either as a polynomial function on $V$ or as a differential operator on $C^{\infty}(V^*;\mathbb{C})$. To distinguish the two cases we shall denote $P^{op}$ the element $P$ interpreted as a differential operator. 
The following is an ubiquitous generalization of the Gaussian integration formula \eqref{gauss}:
\begin{proposition}[Fourier transforms of Gaussian functions]\label{fourier}
\begin{equation}\label{gauss2}
\int_{V}P(x)\exp(-\frac{1}{2}Q(x))\dd\mu=\frac{(2\pi)^{n/2}}{\sqrt{\det(Q,\mu)}}\big(P^{op}\exp(\frac{1}{2}Q^{-1})\big)|_0
\end{equation}
\end{proposition}

If $K\subset F$ is a $k$-dimensional subspace one has $|\Lambda(F)|=|\Lambda(K)|\otimes_{\mathbb{R}}|\Lambda(F/K)|$: indeed given bases $b_K$ and $b_{F/K}$ for $K$ and $F/K$ one can construct easily a basis of $b$ of $F$ using the inclusion of $K$ and choosing an arbitrary complement to $K$ (the choice of the complement affects $b$ only up to elements of $SL(F)$).  Thus given densities $\mu$ and $\mu_F$ on $F$ and $K$ respectively, the \emph{quotient} density $\mu_{F/K}$ on $F/K$ is defined so that $\mu=\mu_K\otimes \mu_{F/K}$.

Suppose that the quadratic form $Q$ is degenerate. Then denoting by $K$ its kernel, we can apply the formula \eqref{gauss} to the reduced quadratic form $\ba{Q}$ on $F/K$. To do that, we need to fix a density $\mu_K$ on $K$. Then, we can set 
$I(Q,\mu_{K})=\int_{F/K} \exp(-\frac{1}{2}\ba{Q}(x)) \dd \mu_{F/K}(x)$.

This new integral can be computed without considering the quotient by the following perturbative argument. Let $Q'$ be a quadratic form on $F$ with positive real part and such that it is non-degenerate on $K$ . Then, for $\epsilon$ positive and small enough, the quadratic form $Q+\epsilon Q'$ is non-degenerate on $K$ and on $F$. Moreover we have

$$I(Q,\mu_{K})=\lim_{\epsilon\to 0} \frac{I(Q+\epsilon Q')}{I(Q'|_K)}.$$
As a consequence , if $\alpha\in \R$ is positive, it holds $I(\alpha Q,\mu_k)=\alpha^{\frac{k-n}{2}}I(Q,\mu_k)$. 
Suppose now that $F=F_1\oplus F_2$, that $Q$ is non-degenerate on $F_1$ and that both $F$ and $F_1$ are equipped with densities $\mu$ and $\mu_{1}$ (consequently $F_2$ inherits a density $\mu_2$). Let $Q_1$ be the restriction of $Q$ to $F_1$ and let $A:F_2\to F_1$ be defined by $Q(x,y)=Q_1(x,Ay),\forall y\in F_2, \forall x\in F_1$. Then, setting $Q'(y)=Q(y,y)-Q(Ay,Ay)$ it holds:
\begin{equation}
I(Q)=I(Q_{1})I(Q') 
 \end{equation} 
 
Given a manifold $M$ equipped with a density $\mu$ one defines $\int_M f d\mu\in \mathbb{C},\ \forall f\in C^{\infty}(M)$ in the natural way. In particular, a compact Lie group $G$ can be equipped by the $G$-invariant density $\mu_{H}$ defined by $\int_{G} 1d\mu_H=1$.
If a $G$ acts freely on a manifold $M$ equipped with a density $\mu$ preserved by $G$, one defines naturally a density $\mu_{M/G}$ on the quotient. 
\begin{example}[The euclidean density on $\cal{G}$]
On $\cal{G}$ equipped with the scalar product $|\xi|^2=-\frac{1}{2}\tr(\xi^2)$, let $\mu_{euc}$ be the density whose value on a orthonormal basis is $1$. Then $\mu_{euc}$ coincides with the density induced by the identification of $\cal{G}$ with the tangent space to $S^3\subset \C^2$ at a point and the Haar density is $\mu_H=\frac{\mu_{euc}}{2\pi^2}$. 
\end{example}
\begin{example}[The Hopf fibration]\label{quotientdensity}
Let $S^{3}=$SU(2) be the unit sphere in $\mathbb{C}^2$ and let $\P^1$ be the quotient by the diagonal action of $S^{1}$. Let $\mu^{\rm SU(2)}_{H}, \mu^{S^1}_{H}$ and $\mu^{\P^1}_{H}$ be the Haar densities on SU(2) and the quotient density on $\P^1$ respectively. Let also $\mu^{\rm SU(2)}_{euc},\mu^{\P^1}_{euc}$ be the euclidean densities. Then $\mu^{\rm SU(2)}_{H}=\frac{1}{2\pi^2}\mu^{\rm SU(2)}_{euc}$ and $\mu^{\P^1}_{H}=\frac{1}{\pi}\mu^{\P^1}_{euc}$.
\end{example}

\section{Around the proof of Theorem \ref{teo:abelianwestbury}}\label{sub:determinants}
The following are general well known facts which we shall apply to interpret topologically some of the determinants we will be dealing with.
 \subsection{Combinatorial interpretation of determinants}
Let $G^{or}$ be a graph whose vertices $v_1,\ldots v_n$ are connected by oriented edges $e_{ij}:v_i\to v_j$ whose weights are the entries $m_{ij}$ of a $n\times n$ matrix $M$ (the diagonal terms correspond then to loops).
The following is a standard well-known fact:
\begin{lemma}
$\det(M)=(-1)^n\sum_{c} (-1)^{\# c}w(c)$ where $c$ runs over all the oriented curves embedded in $G^{or}$ and touching each vertex exactly once, $w(c)$ is the product of the weights $m_{ij}$ of the oriented edges in $c$ and $\# c$ is the number of connected components of $c$.
\end{lemma}
Let now $G$ be the graph whose vertices $v_1,\ldots v_n$ are connected by exactly one unoriented edge $e_{i,j}$ and suppose that $M$ is a matrix such that $m_{ii}=0\  \forall i$. Given a connected oriented curve $c$ embedded in $G$ connecting vertices $v_{i_1}\to v_{i_2}\to \cdots v_{i_k}\to v_{i_1}$ let $w(c)=m_{i_1,i_2}\cdots m_{i_k,i_1}$ and if $c$ is a disconnected union of disjoint oriented curves $c_1,\ldots c_k$, let $w(c)=\prod_{i} w(c_i)$. Also let a \emph{dimer} be disjoint union of edges $d=e_{i_i,j_1}\sqcup \cdots \sqcup e_{i_k,j_k}$, and let $w(d)=m_{i_1,j_1}m_{j_1,i_1}\cdots m_{i_k,j_k}m_{j_1,i_k}$ and $\# d=k$. By a {\it configuration of curves and dimers} we will from now on mean a disjoint union $c\cup d$ of dimers and oriented curves embedded in $G$ such that each vertex is contained exactly in one component of $c\cup d$; its weight will be $w(c\cup d)=w(c)w(d)$; similarly a {\it configuration of dimers} will be a configuration of curves and dimers containing no curves. We shall denote Conf($G$) the set of configurations of curves and dimers on $G$ and DConf($G$) the set of dimer configurations. 
Then the following holds: 
\begin{corollary}\label{cor:determinants}
$\det(M)=(-1)^n\sum_{c\cup d\in {\rm Conf}(G)} (-1)^{\# c+\# d}w(c)w(d)$
\end{corollary}
Finally if $M^t=-M$ then $\det(M)=\pf(M)^2$ and $\pf(M)$ can be interpreted as counting the dimer configurations in $G$ (see \cite{Kas}):
\begin{theorem}\label{teo:dimers}
$Pf(M)=\sum_{d\in {\rm DConf}(G)} \pm \sqrt{|w(d)|}$
\end{theorem}
In the above theorem, the square root is due only to our definition of $w(d)$ while the choice of the signs is in general a delicate matter (see \cite{Kas}); in our specific cases it will be quite easy to determine it.

\subsection{Proof of Theorem \ref{teo:abelianwestbury}}
Let $\Gamma$ be a planar spin network equipped with a holonomy with values in the diagonal matrices of SL$_2\C$ and let us use the notation introduced in Subsection \ref{sub:abeliancase}. Theorem \ref{teo:abelianwestbury} descends directly from Corollary \ref{cor:determinants} and Proposition \ref{prop:countingcurves}: here below the details.

Let $\Gamma'$ be the graph obtained from $\Gamma$ by blowing up the vertices; each curve $\gamma\subset \Gamma$ can be lifted in a natural way to one in $\Gamma'$ (which we will keep calling $\gamma$) and the connection $\psi$ may be lifted to $\psi$ on $\Gamma'$ so that the holonomy on a curve and its lift coincide (see Figure \ref{fig:blowupholonomy}).
\begin{figure}
\centering
  \def\svgwidth{\columnwidth}
 \executeiffilenewer{holonomies.svg}{holonomies.pdf}%
 {inkscape -z -D --file=holonomies.svg %
 --export-pdf=holonomies.pdf --export-latex}%
 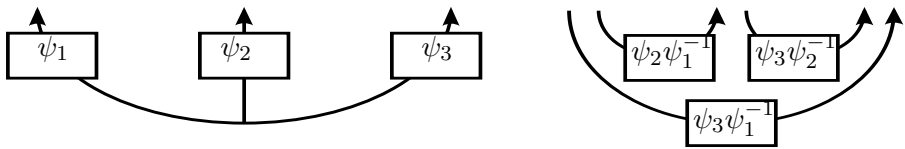%

  \caption{Blowing up with holonomies\label{fig:blowupholonomy}}
\end{figure}

For any half-edge $h$ of $\Gamma$, the space $F_h$ has a standard basis whose corresponding coordinates are $z_h,w_h$. We will say that a basis element has \emph{type} $z$ or $w$. Denote by $W$ the matrix of $\frac{1}{i}(P+Q_{\psi})$ in this basis. Remark that, since $\psi$ is diagonal, the coefficient $W_{i,j}$ does not vanish only if $i,j$ correspond to adjacent vertices of $\Gamma'$ with distinct types.
Recall that the set of half-edges is ordered and put first the type $z$ basis, then the type $w$ basis. Then $W$ is then antidiagonal by blocks. Denote by $W^1,W^2$ the matrices indexed by $H$ defined respectively by $W^1_{g,h}=W_{z_g,w_h}$ and $W^2_{g,h}=W_{w_g,z_h}$. We have  $\det(W)=\det(W^1)\det(W^2)$. 

Using the presentation of $\Gamma'$ in $\mathbb{R}^2$ induced by that of $\Gamma$ (Figure \ref{graphe}) we see that the edges corresponding to adjacent angles in $\Gamma$ form an angle of $0$ degrees (so a cusp) pointing in the upwards direction in $\Gamma'$ (see Figure \ref{fig:blowupholonomy}); we will call such edges \emph{internal} and the other ones \emph{external}.

\begin{remark}\label{rem:coeffs}
If $e$ is an edge of $\Gamma'$, let $l_e,r_e$ be the half-edges respectively at the left and right endpoint of $e$.  
If $e$ is an external edge of $\Gamma'$ oriented from left to right, then the entry of $W_1$ corresponding to $e$ is $1$ and that corresponding to $e$ equipped with the opposite orientation is $-1$: indeed both entries come from the matrix expressing $-iP$. Similarly, if $\alpha$ is an internal edge of $\Gamma'$ oriented from left to right, then the entry of $W_1$ corresponding to $\alpha$ is $(t_{l_\alpha}^{-1}t_{r_\alpha})X_{\alpha}$ and that corresponding to the opposite orientation is $-(t_{l_\alpha}t_{r_\alpha}^{-1})X_{\alpha}$: both entries come from the matrix expressing $-iQ_\psi$. 
\end{remark}

Let us denote $\ba{t_h}=t_h^{-1}$. This extends to an automorphism of the ring of coefficients. We check that $W^2=-\ba{W^1}=(W^1)^t$ hence, $\det(W^2)=\ba{\det(W^1)}=\det(W^1)$ because there is an even number of half-edges, so, by Theorem \ref{thm1}, to prove Theorem \ref{teo:abelianwestbury} it is sufficient to interpret $\det(W^1)$ in terms of traces of curves. 
To this purpose, remark that $W^1$ has $0$ on the diagonal and we can apply Corollary \ref{cor:determinants} with $G=\Gamma'$ and $M=W^1$.

Now remark that if a vertex $v$ of $\Gamma'$ is contained in a curve $c$ embedded in $\Gamma'$ then either $c$ contains an internal  and an external edge containing $v$ or it contains two internal edges containing $v$ : we will call the latter kind of vertices ``cusps" of $c$; also a dimer covering an internal edge will be called {\it internal}.

\begin{proposition}\label{prop:abeliancase}
Let $c\cup d$ be a configuration of curves and dimers containing each vertex of $\Gamma'$ exactly once. Then:
\begin{enumerate}
\item There exists a unique set $L$ of disjoint arcs in $\Gamma'\setminus c$, whose boundary vertices coincide with the set of cusps of $c$.
\item Let $d'=d\setminus L$: the set of internal dimers of $d'$ can be joined by external edges of $\Gamma'$ to form a unique embedded (possibly disconnected) curve $c'\subset \Gamma'\setminus (c\cup L)$.
\end{enumerate}
\end{proposition}
{\bf Proof.}
The first statement is proved by remarking that if $e$ is an edge of $\Gamma'$ whose endpoint is a cusp of a configuration, then its other endpoint can only be either another cusp (in which case set $l_i=e$) or must be contained in an internal dimer. By iterating this argument for the edge at the other endpoint of that internal dimer one eventually constructs a path $l_i$ which must end at another cusp of $c$. The second statement is proved by a similar inspection.
\qed

There is a natural map $\pi:C(\Gamma)\to {\rm Conf}(\Gamma')$: define $\pi(\gamma)$ as $c\cup d$ where $c$ is the oriented curve in $\Gamma'$ formed by the edges and angles of $\Gamma$ contained in $\gamma$ exactly once and $d$ is the dimer formed by all the angles contained twice in $\gamma$ and all the edges not contained in $\gamma$.

\begin{proposition}\label{prop:countingcurves}
The map $\pi$ is a bijection and letting $\pi(\gamma)=c\cup d$, it holds: $$(-1)^{\#c+\#d}w(c\cup d)=(-1)^{{\rm cr}(\gamma)}\tr(\gamma)X^{\gamma}$$ where cr$(\gamma)$ is the number of crossings of $\gamma$. \end{proposition}
\begin{proof}
We construct the reciprocal map $s:{\rm Conf}(\Gamma')\to C(\Gamma)$  by associating to a configuration $c\cup d$ the unique oriented curve $\gamma$ which passes once exactly over the edges contained in $c$ (and with the same orientation as $c$), twice over all the edges contained in $L$ (see Proposition \ref{prop:abeliancase} applied to $c\cup d$) and twice over all the edges contained in $c'$ (in that case we consider the curve without crossings). This proves our first claim.

Set $\pi(\gamma)=c\cup d$ then for any connected component $\delta$ of $\gamma$, we define $w(\delta)=\prod_{k=0}^n W^1_{h_{k},h_{k+1}}$ where $\delta$ visits the half-edges $h_0,\ldots,h_n,h_{n+1}=h_0$ in that order.
We first claim that the following formula holds: $$(-1)^{\# c+\#d}w(c)w(d)=(-1)^{\#\gamma}w(\gamma).$$
 We check this formula by looking first at the $X_\alpha$ terms, then the $t_h$ terms and then the signs. 
\begin{itemize}
\item[-] The monomial in $X_{\alpha}$ occuring in the right hand side is $X^\gamma$, where each angle appears as many times it is visited by $\gamma$. On the left hand side, each time an internal edge $\alpha$ is visited by $c$ produces a monomial $X_\alpha$ whereas each time a dimer occupies an internal edge $X_{\alpha}$ produces a monomial $X_{\alpha}^2$. As internal edges occupied by a dimer are visited twice by $\gamma$, these monomials coincide.
\item[-] By Remark \ref{rem:coeffs}, if an internal edge is visited twice in opposite directions, the resulting monomial in $t_h$ is equal to 1. Hence, monomials occur only for internal edges visited once. One sees that the set of internal edges visited once by $\gamma$ and $c\cup d$ coincide (with orientation) hence the resulting monomial is the same.
\item[-] By Remark \ref{rem:coeffs}, for any individual dimer $d_0$ one has $w(d_0)=-1$. Hence, one has $(-1)^{\# d}w(d)=1$ (we drop here the $X_\alpha$ and $t_h$ parts which have already been taken care of). In order to show $(-1)^{\#\gamma}w(\gamma)=(-1)^{\# c}w(c)$, we recall that to build $\gamma$ from $c$, one needs to add the double arcs $L$ and the double curve $c'$. For the latter, one sees that the weight is 1 because the number of curves is even and each curve goes through a even number of edges. Now, any added double arc contains exactly one more external edge than internal ones and thus produces a -1 factor. On the other hand, either it merges two components of $c$ or it splits one component in two hence the formula is proven.
\end{itemize}
Let us show finally that $(-1)^{\#\gamma}w(\gamma)=(-1)^{{\rm cr}(\gamma)}\tr(\gamma)X^\gamma$. We only have to check the sign, and one can suppose that $\gamma$ is connected. Then one has to show that Sign $w(\gamma)=(-1)^{{\rm cr}(\gamma)+1}=(-1)^{{\rm wind}(\gamma)}$ where ${\rm wind}(\gamma)$ is the winding number of $\gamma$ (here we use the equality ${\rm wind}(\gamma)={\rm cr}(\gamma)+1$ mod 2). 

The lift of $\gamma$ to an oriented curve immersed in $\Gamma'$ meets alternatively an internal edge and an external one: by Remark \ref{rem:coeffs} if both are covered in the same direction (e.g right to left) then the contribution to $w(\gamma)$ has positive coefficient and if not, then the contribution is negative. On the other hand, the tangent vector to $\gamma$ rotates either of $0$ or $\pm 2\pi$ depending on the same condition. Hence, one has Sign $w(\gamma)=(-1)^{{\rm wind}(\gamma)}$ and thus the formula of the proposition is proved. \end{proof}

\subsection{Recovering Westbury's theorem}\label{sub:recoverwestbury}
Theorem \ref{theo:westbury} can be recovered from Theorem \ref{teo:abelianwestbury} by remarking that if $t_h=1\  \forall\  h\in H$ then the equalities $W^2=-\ba{W^1}$ and $W^2=(W^1)^t$ imply $(W^1)^t=-W^1$ and so $\det(W^1)=\pf(W^1)^2$. Hence we must show that $\pf(W^1)=\sum_{c\subset \Gamma} \prod_{\alpha\subset c } X_\alpha$. 
By Theorem \ref{teo:dimers} it holds $$\pf(W^1)=\sum_{d\in {\rm DConf}(\Gamma')} \pm \sqrt{|w(d)|}=\sum_{d\in {\rm DConf}(\Gamma')} \pm\prod_{\alpha \subset d} X_\alpha$$ and it is straightforward to check that for each $d\in {\rm DConf}(\Gamma')$ the monomial $\sqrt{|w(d)|}$ equals $\pm w(c(d))$ for a unique (possibly empty or disconnected) curve $c(d)$ embedded in $\Gamma$; similarly $\forall \gamma\subset \Gamma$ it exists exactly one $d\in {\rm DConf}(\Gamma')$ such that $c(d)=\gamma$. 
So we have $\pf(W^1)=\sum_{\gamma\subset \Gamma}s(\gamma)\prod_{\alpha\subset \gamma } X_\alpha$ (with $s(\gamma)=\pm 1$) and we are left to show that $s(\gamma)=1\ \forall \gamma\subset \Gamma$. Let us first remark that by Theorem \ref{teo:abelianwestbury} the coefficient of $w(\gamma_1\sqcup \cdots \sqcup \gamma_k)=\prod_{\alpha\subset \gamma } X_\alpha$ in $\det(W^1)$ is $2^k$. We now use the equality $\det(W^1)=\pf(W^1)^2$ and argue by induction on $k$. If $k=1$ the coefficient of $w(\gamma)$ in $\det(W^1)$ is $2=2s(\gamma)$ and so $s(\gamma)=1$. Now suppose $s(\gamma)=1$ for all $\gamma\subset \Gamma$ such that $\# \gamma\leq k$; then given $\gamma=\gamma_1\sqcup \cdots \sqcup \gamma_{k+1}$, the coefficient of $w(\gamma)$ in $\det(W^1)$ is:
$$2^{k+1}=2s(\gamma)+2\sum_{\substack{I\subset \{1,\ldots k+1\}\\ 1\leq \#I\leq \frac{k+1}{2}}} s(\bigsqcup_{i\in I} \gamma_i)s(\bigsqcup_{i\notin I} \gamma_i)=2(s(\gamma)+\frac{1}{2}(2^{k+1}-2))$$
and so $s(\gamma)=1$.

\end{document}